\documentclass{amsart}
\usepackage[letterpaper,centering]{geometry}
\usepackage{hyperref}
 \hypersetup{
     linktocpage=true,  
     colorlinks=true,
     linkcolor=blue,
     citecolor=blue,
     urlcolor=blue,
 }
\makeatletter
\newcommand\org@maketitle{}
\newcommand\@authors{}
\let\org@maketitle\maketitle
\def\maketitle{%
	\let\@authors\authors
	\nxandlist{; }{ and }{; }\@authors
	\hypersetup{
		pdftitle={\@title},
                pdfauthor={\@authors},
                pdfsubject={\subjclassname. \@subjclass},
		pdfkeywords={\@keywords}
	}%
	\org@maketitle
}
\makeatother

\usepackage[alphabetic,msc-links,initials]{amsrefs}
\usepackage{amssymb}
\usepackage{mathtools}
\usepackage{enumitem}
\usepackage{rsfso}
\usepackage{microtype}
\usepackage{doi}
\renewcommand{\PrintDOI}[1]{\doi{#1}}

\let\arXiv\arxiv

\numberwithin{equation}{section}
\newtheorem{maintheorem}{Theorem}

\newtheorem{theorem}{Theorem}[section]
\newtheorem{lemma}[theorem]{Lemma}
\newtheorem{proposition}[theorem]{Proposition}
\newtheorem{corollary}[theorem]{Corollary}
\theoremstyle{definition}
\newtheorem{definition}[theorem]{Definition}
\theoremstyle{remark}

\newtheorem{example}[theorem]{Example}

\newcommand{\mfa}{\mathfrak{a}}
\newcommand{\cK}{{\mathcal K}}
\newcommand{\cM}{{\mathcal M}}

\newcommand{\e}{\varepsilon}
\newcommand{\g}{\gamma}

\newcommand{\la}{\lambda}

\newcommand{\Ld}{\Lambda}

\newcommand{\D}{\nabla}

\newcommand{\pa}{\partial}

\newcommand{\dv}{\operatorname{div}}
\newcommand{\R}{\mathbb{R}}
\newcommand{\vp}{\varphi}

\newcommand{\al}{\alpha}
\newcommand{\be}{\beta}

\newcommand{\si}{\sigma}
\newcommand{\sm}{\setminus}

\newcommand{\om}{\omega}

\newcommand{\K}{\widetilde{K}}

\newcommand{\wDu}{\widehat{\D u}}
\newcommand{\wDv}{\widehat{\D v}}
\newcommand{\wDw}{\widehat{\D w}}
\newcommand{\wDU}{\widehat{\D U}}

\newcommand{\loc}{\mathrm{loc}}

\newcommand{\dist}{\operatorname{dist}}
\newcommand{\mean}[1]{\langle{#1}\rangle}

\def\Xint#1{\mathchoice
  {\XXint\displaystyle\textstyle{#1}}%
  {\XXint\textstyle\scriptstyle{#1}}%
  {\XXint\scriptstyle\scriptscriptstyle{#1}}%
  {\XXint\scriptscriptstyle\scriptscriptstyle{#1}}%
  \!\int}
\def\XXint#1#2#3{{\setbox0=\hbox{$#1{#2#3}{\int}$}
    \vcenter{\hbox{$#2#3$}}\kern-.5\wd0}}

\def\dashint{\Xint-}

\mathtoolsset{centercolon}

\allowdisplaybreaks[4]

\makeatletter
\@namedef{subjclassname@2020}{%
  \textup{2020} Mathematics Subject Classification}
\makeatother

\author{Seongmin Jeon}
\address{Department of Mathematics, KTH Royal Institute of Technology, Stockholm, Sweden}
\email{seongmin@kth.se}
\author{Arshak Petrosyan}
\address{Department of Mathematics, Purdue University, West Lafayette,
  IN 47907, USA}
\email{arshak@purdue.edu}
\thanks{The second author is supported in part by NSF Grant DMS-1800527.}

\title[Almost minimizers for parabolic thin obstacle problem]{Regularity of almost minimizers for the parabolic thin obstacle problem}

\begin{document}
\begin{abstract}
In this paper, we study almost minimizers for the parabolic thin obstacle (or Signorini) problem with zero obstacle. We establish their $H^{\si,\si/2}$-regularity for every $0<\si<1$, as well as $H^{\be,\be/2}$-regularity of their spatial gradients on the either side of the thin space for some $0<\be<1$. A similar result is also obtained for almost minimizers for the Signorini problem with variable H\"{o}lder coefficients.
\end{abstract}
\maketitle
\tableofcontents

\section{Introduction and main results}

\subsection{The parabolic thin obstacle (or Signorini) problem}
Let $\Omega$ be a domain in $\R^n$, $n\ge2$, and $\mathcal M$ be a smooth $(n-1)$-dimensional manifold that divides $\Omega$ into two parts: $\Omega\setminus\mathcal{M}=\Omega^+\cup \Omega^-$. For $T>0$, we set $\Omega_T:=\Omega\times(0,T]$, $\cM_T:=\cM\times(0,T]$ (\emph{the thin space}), and $(\pa\Omega)_T:=\pa\Omega\times(0,T]$. Let also $\psi:\cM_T\to\R$ (\emph{the thin obstacle}), $\psi_0:\Omega\times\{0\}\to\R$ (\emph{the initial value}), and $g:(\pa\Omega)_T\to\R$ (\emph{the boundary value}) be prescribed functions satisfying the compatibility conditions: $\psi_0\ge\psi$ on $\cM\times\{0\}$, $g\ge\psi$ on $(\cM\cap\pa\Omega)\times(0,T]$, and $g=\psi_0$ on $\pa\Omega\times\{0\}$.

We then say that a function $u\in W^{1,0}_2(\Omega_T)$ (see Sec~\ref{sec:notation} for notations) is a solution of the \emph{parabolic thin obstacle} (or \emph{Signorini}) \emph{problem} in $\Omega_T$, if it satisfies the variational inequality \begin{align}
    \label{eq:par-sig-var-ineq}&\int_{\Omega_T}\D u\D(v-u)+\pa_tu(v-u)\ge0\quad\text{for any }v\in \cK,\\
    &u\in \mathcal{K},\ \pa_tu\in L^2(\Omega_T),\ u(\cdot,0)=\psi_0 \text{ on } \Omega,
\end{align}
where $\mathcal K=\{v\in W^{1,0}_2(\Omega_T):v\ge\psi\text{ on }\cM_T,\ v=g\text{ on }(\pa\Omega)_T\}$. It is known that the solution $u$ satisfies 
\begin{align*}
    \Delta u-\pa_tu=0 & \text{ in }\Omega_T\setminus\cM_T,\\
    u\ge\psi,\ \pa_{\nu^+}u+\pa_{\nu^-}u\ge0,\ (u-\psi)(\pa_{\nu^+}u+\pa_{\nu^-}u)=0&\text{ on }\cM_T,\\
    u=g&\text{ on }(\pa\Omega)_T,\\
    u(\cdot,0)=\psi_0&\text{ on }\Omega\times\{0\},
\end{align*}
to be understood in a certain weak sense, where $\nu^\pm$ are the outer unit normal to $\Omega^\pm$ on $\cM$. 

Problems with unilateral constraints of this type go back to the original problem of Signorini in elastostatics \cite{Sig59}, but also appear in many applications ranging from math biology (semipermeable membranes) and boundary heat control \cite{DuvLio76} to math finance (American options) \cite{ConTan04}. Since the condition on the thin space $\cM_T$ changes from the Dirichlet-type condition $u=\psi$ to the Neumann-type condition $\pa_{\nu^+}u+\pa_{\nu^-}u=0$ as one crosses the apriori unknown interface
$$
\Gamma(u)=\partial_{\cM_T}\{u=\psi\}
$$
called \emph{the free boundary}, even the regularity of the solutions is a challenging problem. It has been known since the '80s that $\nabla u$ is parabolically H\"older continuous, i.e., $\nabla u\in H^{\beta,\beta/2}$ for some $0<\beta<1$, see \cites{Ura85, ArkUra96}. More recent results, based on the emergence of new monotonocity formulas, starting with \cites{AthCaf04, AthCafSal08,GarPet09} in the elliptic case and \cite{DanGarPetTo17} in the parabolic case, established the optimal regularity $\nabla u\in H^{1/2,1/4}$ at least when $\cM$ is flat; see also \cites{PetZel19,BanSVGZel17,AthCafMil18,BanDanGarPet20,BanDanGarPet21} for some further results. 

\subsection{Almost minimizers}
In this paper we introduce and investigate the almost minimizers related to the parabolic Signorini problem described above. For this purpose, we observe that an equivalent formulation of \eqref{eq:par-sig-var-ineq} is $$
\int_{\Omega_T}|\D u|^2+2\pa_tu(u-v)\le \int_{\Omega_T}|\D v|^2\quad\text{for any }v\in \cK.
$$
Based on this fact, we give the following definition of almost minimizers.

\begin{definition}\label{def:alm-min-par-Sig}
Given $r_0>0$ and a gauge function $\omega(r)$ defined on $(0,r_0)$\footnote{A gauge function $\omega:(0,r_0)\to[0,\infty)$ is a nondecreasing function with $\omega(0+)=0$.}, we say that a function $u\in W^{1,1}_2(\Omega_T)$ is an \emph{almost minimizer for the parabolic Signorini problem}, if $u\ge\psi$ on $\cM_T$ and for any parabolic cylinder $Q_r(z_0)=Q_r(x_0,t_0)\Subset\Omega_T$ with $0<r<r_0$, we have \begin{align}\label{eq:par-alm-min-def}
\int_{Q_r(z_0)}(1-\omega(r))|\D u|^2+2\pa_tu(u-v)\le (1+\omega(r))\int_{Q_r(z_0)}|\D v|^2
\end{align}
for any $v\in W^{1,0}_2(Q_r(z_0))$ with $v\ge \psi$ on $Q_r(z_0)\cap\cM_T$ and $v-u\in L^2(t_0-r^2,t_0;W_0^{1,2}(B_r(x_0)))$.
\end{definition}

Almost minimizers as above can be seen as perturbations of the solution of the parabolic Signorini problem. They can also be viewed as some very weak form of solutions, strong enough to retain some of the core properties of solutions. The main complication in the study of almost minimizers is that they do not satisfy a partial
differential equation. Thus, when we deal with almost minimizers we can only use comparison \eqref{eq:par-alm-min-def} with appropriately chosen competitors.

The notion of almost minimizers (without obstacle) was first introduced by Anzellotti \cite{Anz83}. Concerning almost minimizers in the framework of the thin obstacle problem, its time-independent version was studied by the authors in \cite{JeoPet21}. The results were extended to almost minimizers for the energy functional with variable H\"older coefficients by the authors and Smit Vega Garcia in \cite{JeoPetSVG20}. In this paper we extend some of those results to the parabolic setting. Regarding almost minimizers in the time-dependent setting, Habermann \cite{Hab14} considered vector-valued parabolic almost minimizers related to Dirichlet $p$-energy functionals (without obstacle).

%%%%%%%%%%%%%%%%%%%%%%%%%%%%%%%%%%%%%%%%%%%%%%%%

\subsection{Almost minimizers with variable coefficients}

One of the motivations and advantages of developing methods that work with almost minimizers is that they allow the extension of the results to the case of variable coefficients with the moduli of continuity controlled by the gauge function $\omega(r)$. As we will be working with $\omega(r)=r^\alpha$, $0<\alpha<1$, for the most of the paper, we extend he notion of the almost minimizers in Definition~\ref{def:alm-min-par-Sig} to include parabolic H\"older coefficients in the energy functional.

Let $A(z)=(a_{ij}(z))$ be an $n\times n$ symmetric uniformly elliptic matrix,
$H^{\alpha, \al/2}$-regular as a function of $z\in \Omega_T$, for some
$0<\alpha<1$, with ellipticity constants $0<\lambda\leq1\leq\Lambda<\infty$:
$$
\lambda|\xi|^2\leq \langle{A(z)\xi,\xi}\rangle\leq
\Lambda|\xi|^2,\quad z\in \Omega_T,\ \xi\in\R^n.
$$
For $z_0=(x_0,t_0)$, define $n$-dimensional ellipsoids $$
E_r(z_0):=A^{1/2}(z_0)(B_r)+x_0
$$
and $(n+1)$-dimensional elliptic cylinders
$$
F_r(z_0):=E_r(z_0)\times(t_0-r^2,t_0].
$$
By ellipticity of $A(z_0)$, we have
\begin{gather*}
B_{\la^{1/2}r}(x_0)\subset E_r(z_0)\subset B_{\Lambda^{1/2}r}(x_0)\\
Q_{\la^{1/2}r}(z_0)\subset F_r(z_0)\subset Q_{\Lambda^{1/2}r}(z_0).
\end{gather*}

% Using those variable coefficients $A$ and the relevant cylinders $F_r(z_0)$, we generalize the notion of almost minimizers for the parabolic Signorini problem \eqref{eq:par-alm-min-def} to the variable coefficient case.

We are now ready to extend the notion of the almost minimizers to the variable coefficient case.

\begin{definition}\label{def:alm-min-par-A-Sig}
We say that a function $U$ is an \emph{almost minimizer for the parabolic $A$-Signorini problem} in $\Omega_T$ if $U\in W^{1,1}_{2}(\Omega_T)$, $U\ge \psi$ on $\cM_T$, and 
\begin{equation}
    \label{eq:par-var-alm-min-def}
    \int_{F_r(z_0)}(1-\omega(r))\mean{A\D U,\D U}+2\pa_tU(U-V)\le (1+\omega(r))\int_{F_r(z_0)}\mean{A\D V,\D V}
\end{equation}
for any $F_r(z_0)\Subset \Omega_T$ and $V\in W^{1,0}_2(F_r(z_0))$ with $V\ge \psi$ on $F_r(z_0)\cap \cM_T$ and $V-U\in L^2(t_0-r^2,t_0;W^{1,2}_0(E_r(z_0)))$.
\end{definition}

We explicitly note that $W^{1,1}_2$ weak solutions of the parabolic $A$-Signorini problem
\begin{align*}
    -\dv(A\D U)+\pa_tU=0&\quad\text{in }\Omega_T\\*
    \mean{A\D U,\nu^+}+\mean{A\D U,\nu^-}\ge 0,\ U\ge 0,\ U(\mean{A\D U,\nu^+}+\mean{A\D U,\nu^-})=0&\quad\text{on }\cM_T
\end{align*}
are almost minimizers in the sense above and thus our results are applicable also to them. If one perturbs the PDE in the first line above, e.g.\ by adding a linear transport term with appropriate conditions on the coefficients, then the solutions are still almost minimizers in the sense above, at least when the thin space $\cM$ is flat. The details are given in Appendix~\ref{sec:appen-ex}. 

% Notice that when the coefficient $A(z)$ is identically the identity matrix $I$, the almost minimizer in \eqref{eq:par-var-alm-min-def} is the same as in \eqref{eq:par-alm-min-def}.

%%%%%%%%%%%%%%%%%%%%%%%%%%%%%%%%%%%%%%%%%%%%%%%%%%%

\subsection{Main results and structure of the paper}

In this paper we are interested in local regularity results for almost minimizers. Thus we assume that the domain $\Omega_T\subset \R^n\times\R$ is the parabolic cylinder $Q_1$. Due to the technical nature of the problem, we consider only the case when the thin space $\cM_T$ is $Q'_1$ (flat thin space), the thin obstacle $\psi$ is identically zero (zero thin obstacle), and the gauge function $\omega(r)=r^\al$ for some $0<\al<1$.

\medskip 

We now state our main results on the regularity of almost minimizers. Our first theorem is as follows.

\begin{maintheorem}
  \label{mthm:I}
 Let $u$ be an almost minimizer of the parabolic Signorini problem in $Q_1$. Then
  \begin{enumerate}
  \item $u\in H^{\si,\si/2}(Q_1)$ for any $0<\si<1$;
  \item $\D u\in H^{\be,\be/2}_{\loc}(Q_1^\pm\cup Q'_1)$ for some $\be=\be(\al,n)>0$.
  \end{enumerate}
\end{maintheorem}

For the proof, we mainly follow the idea in its elliptic counterpart \cite{JeoPet21}. However, the situation is more complicated in our setting. In \cite{JeoPet21}, a crucial step in the study of H\"{o}lder continuity is the following version of concentric ball estimates \begin{align}
    \label{eq:ell-ball-est}
    \int_{B_\rho(x_0)}|\D u|^2\le C\left[\left(\frac\rho r\right)^n+r^\al\right]\int_{B_r(x_0)}|\D u|^2,\quad 0<\rho<r.
\end{align}
It gives Morrey-type estimates, which in turn imply the H\"{o}lder regularity by Morrey space embedding theorem. Unfortunately, this approach is inapplicable to our parabolic case, since the Morrey space embedding, involving only spatial gradients, does not hold in general for functions defined in the parabolic space. The main reason is that the Poincar\'e inequality, which contains merely spatial derivatives, is not satisfied for arbitrary parabolic functions. To overcome this difficulty, we obtain concentric cylinder estimates with functionals 
$$
\vp_{z_0}(r,u):=\int_{Q_r(z_0)}r^{n+4}|\D u|^2+\int_{Q_r(z_0)\times Q_r(z_0)}|u(z)-u(w)|^2\,dzdw
$$ 
instead of standard Dirichlet integrals in \eqref{eq:ell-ball-est}, see Proposition~\ref{prop:par-alm-min-Mor-est-thin}.

\medskip 

The next result concerns the regularity of almost minimizers with variable coefficients.

\begin{maintheorem}
  \label{mthm:II}
 Let $U$ be an almost minimizer of the parabolic $A$-Signorini problem in $Q_1$, with coefficient matrix $A\in H^{\alpha,\alpha/2}(Q_1)$. Then
  \begin{enumerate}
  \item $U\in H^{\si,\si/2}(Q_1)$ for any $0<\si<1$;
  \item $\D U\in H^{\be,\be/2}_{\loc}(Q_1^\pm\cup Q'_1)$ for some $\be=\be(\al,n)>0$.
  \end{enumerate}
\end{maintheorem}
Note that this theorem contains Theorem~\ref{mthm:I} as a particular case when $A(z)\equiv I$. Its proof follows from the combination of the estimates obtained when $A(z)\equiv I$ and the application of the argument in \cite{JeoPetSVG20}. It is also worth noting that Theorem~\ref{mthm:II} even improves some of the results available for the solutions of the parabolic Signorini problem with variable coefficients. For example, we only require the coefficients $A(z)$ to be $H^{\al,\al/2}$ with arbitrary $0<\al<1$, while stronger assumption $A\in W^{1,1}_q$, with $q>n+2$, is typically required in the existing results in the literature, see e.g.\ \cite{ArkUra96}. 

\medskip

The paper is organized as follows.\\*
\begin{itemize}
\item[-] In Section~\ref{sec:alm-cal} we study almost caloric functions (almost minimizers without the thin obstacle), and obtain their regularity. The growth estimates obtained in this section will be used when we study the regularity for almost minimizers in Sections~\ref{sec:alm-min-holder} and \ref{sec:alm-min-grad-holder}. \item[-] Section~\ref{sec:alm-min-holder} is devoted to proving the H\"older regularity of almost minimizers, the first part of Theorem~\ref{mthm:I} (Theorem~\ref{thm:par-alm-min-holder}).
\item[-] In Section~\ref{sec:alm-min-grad-holder} we take advantage of the estimates obtained in Sections~\ref{sec:alm-cal} and \ref{sec:alm-min-holder} to derive the higher regularity of almost minimizers, the second part of Theorem~\ref{mthm:I} (Theorem~\ref{thm:par-alm-min-grad-holder}).
\item[-] Section~\ref{sec:alm-min-var} is dedicated to generalizing the results in the previous sections to the problem with variable coefficients. We prove Theorem~\ref{mthm:II} (Theorems~\ref{thm:var-par-alm-min-holder} and \ref{thm:var-par-alm-min-grad-holder}) by following the line of \cite{JeoPetSVG20}.
\item[-] Finally, in Appendix~\ref{sec:appen-ex}, we provide an example of an almost minimizer.
\end{itemize}

%%%%%%%%%%%%%%%%%%%%%%%%%%%%%%%%%%%%%%%%%%%%%%%%%

\subsection{Notation}\label{sec:notation}
We use the following notation throughout this paper.

For $x_0\in \R^n$ and $r>0$, we denote 
\begin{align*}
    &B_r(x_0)=\{x\in\R^n\,:\,|x-x_0|<r\}:\text{ Euclidean open ball}\\
    &B_r'(x_0)=\{x\in B_r(x_0)\,:\,x_n=0\}:\text{ thin ball}\\
    &B_r^\pm(x_0)=\{x=(x_1,x_2,\dots,x_n)\in B_r(x_0)\,:\,\pm x_n>0\}.%:\text{ Euclidean half-ball}
    \\
\intertext{For a point $z_0=(x_0,t_0)\in \R^{n}\times\R$, we let}
    &Q_r(z_0)=B_r(x_0)\times(t_0-r^2,t_0]:\text{ parabolic cylinder}\\
    &Q'_r(z_0)=B'_r(x_0)\times(t_0-r^2,t_0]:\text{ thin parabolic cylinder}\\
    &Q_r^\pm(z_0)=B_r^\pm(x_0)\times(t_0-r^2,t_0]%:\text{ parabolic half-cylinder}
    \\
    &\pa_pQ_r(z_0)=\pa B_r(x_0)\times[t_0-r^2,t_0]\cup B_r(x_0)\times\{t_0-r^2\}:\text{ parabolic boundary}.
\end{align*}
When $x_0=0$ or $z_0=0$, we simply write $$
B_r=B_r(0),\quad Q_r=Q_r(0).
$$
In general, for a set $E$ in $\R^n$ or $\R^n\times\R$, we denote
$$
E'=E\cap\{x_n=0\},\quad E^\pm=E\cap\{\pm x_n>0\}.
$$
For points $z_0=(x_0,t_0)$, $z_1=(x_1,t_1)$ in $\R^n\times\R$, we denote \begin{align*}
    &\|z_0\|=\left(|x_0|^2+|t_0|\right)^{1/2}\,:\, \text{parabolic norm}\\
    &d_{\rm par}(z_0,z_1)=\|z_0-z_1\|=\left(|x_0-x_1|^2+|t_0-t_1|\right)^{1/2}\,:\,\text{parabolic distance}
\end{align*}
For a bounded open set $\Omega\subset\R^n$, let $\Omega_T:=\Omega\times(0,T]\subset \R^n\times\R$ for some $T>0$. For a function $u$ defined in $\Omega_T$, we indicate the mean value of $u$ in $\Omega_T$ by $$
\mean{u}_{\Omega_T}=\dashint_{\Omega_T}u=\frac1{|\Omega_T|}\int_{\Omega_T}u.
$$
In particular, when $\Omega_T=Q_r(z_0)$ or $\Omega=Q_r$, we simply write $$
\mean{u}_{z_0,r}=\mean{u}_{Q_r(z_0)},\quad \mean{u}_r=\mean{u}_{0,r}=\mean{u}_{Q_r}.
$$
For $k\in \mathbb N$ and $1\le q\le\infty$, let $L^q(\Omega)$, $W^{k,q}(\Omega)$ and $W_0^{k,q}(\Omega)$ be standard Lebesgue and Sobolev spaces. For some parabolic function spaces, we follow the notation used in \cite{DanGarPetTo17}. We let 
$$
\|u\|_{L^q(\Omega_T)}=\left(\int_{\Omega_T}|u(x,t)|^q\,dxdt\right)^{1/q},
$$
and denote by $W^{1,0}_q(\Omega_T)$, $W^{1,1}_q(\Omega_T)$ the spaces of functions $u$ for which the following norms are finite: \begin{align*}
    &\|u\|_{W^{1,0}_q(\Omega_T)}=\|u\|_{L^q(\Omega_T)}+\|\D u\|_{L^q(\Omega_T)},\\
    &\|u\|_{W^{1,1}_q(\Omega_T)}=\|u\|_{L^q(\Omega_T)}+\|\D u\|_{L^q(\Omega_T)}+\|\pa_tu\|_{L^q(\Omega_T)}.
\end{align*}
Next, we define the parabolic H\"older classes $H^{l,l/2}(\Omega_T)$. For $l=m+\gamma$, $m\in\mathbb N\cup\{0\}$, $0<\g\le 1$, we let \begin{align*}
    [u]_{\Omega_T}^{(0)}&=\sup_{(x,t)\in\Omega_T}|u(x,t)|,\\
    [u]_{\Omega_T}^{(m)}&=\sum_{|\boldsymbol{\al}|+2j=m}[\pa_x^{\boldsymbol{\al}}\pa_t^ju]_{\Omega_T}^{(0)},\\
    [u]_{x,\Omega_T}^{(\gamma)}&=\sup_{(x,t),(y,t)\in\Omega_T}\frac{|u(x,t)-u(y,t)|}{|x-y|^\g},\\
    [u]_{t,\Omega_T}^{(\gamma)}&=\sup_{(x,t),(x,s)\in\Omega_T}\frac{|u(x,t)-u(x,s)|}{|t-s|^\g},\\
    [u]_{x,\Omega_T}^{(l)}&=\sum_{|\boldsymbol{\al}|+2j=m}[\pa_x^{\boldsymbol{\al}}\pa_t^ju]_{x,\Omega_T}^{(\g)},\\
    [u]_{t,\Omega_T}^{(l/2)}&=\sum_{m-1\le |\boldsymbol{\al}|+2j\le m}[\pa_x^{\boldsymbol{\al}}\pa_t^ju]_{t,\Omega_T}^{((l-|\boldsymbol{\al}|-2j)/2)},\\
    [u]_{\Omega_T}^{(l)}&=[u]_{x,\Omega_T}^{(l)}+[u]_{t,\Omega_T}^{(l/2)}.
\end{align*}
Then, $H^{l,l/2}(\Omega_T)$ indicates the space of functions $u$ generated by the norm $$
\|u\|_{H^{l,l/2}(\Omega_T)}=\sum_{k=0}^m[u]_{\Omega_T}^{(k)}+[u]_{\Omega_T}^{(l)}.
$$

%%%%%%%%%%%%%%%%%%%%%%%%%%%%%%%%%%%%%%%%%%%%%

\section{Almost Caloric functions}\label{sec:alm-cal}

In this section, we study almost minimizers without the thin obstacle. We call such functions almost caloric functions, and establish their regularity (Theorem~\ref{thm:alm-cal-reg}). The estimates for almost caloric functions (in particular, Proposition~\ref{prop:alm-cal-Mor-Camp-est}) will be used in the study of almost minimizers with obstacle, for these almost minimizers away from the thin space can be seen as almost caloric functions.

\begin{definition}[Almost caloric property at a point]
We say that a function $u\in W^{1,1}_2(Q_1)$ satisfies the \emph{almost caloric property} at $z_0=(x_0,t_0)$ in $Q_1$ if $$
\int_{Q_r(z_0)}(1-r^\al)|\D u|^2+2\pa_tu(u-v)\le (1+r^\al)\int_{Q_r(z_0)}|\D v|^2
$$
for any $Q_r(z_0)\Subset Q_1$ and for any $v\in W^{1,0}_2(Q_r(z_0))$ with $v-u\in L^2(t_0-r^2,t_0;W^{1,2}_0(B_r(x_0))$.
\end{definition}

\begin{definition}[Almost caloric functions]
We say that $u\in W^{1,1}_2(Q_1)$ is \emph{almost caloric} in $Q_1$ if it satisfies the almost caloric property at every $z_0\in Q_1$.
\end{definition}

We first prove the following growth estimates for caloric functions.

\begin{proposition}\label{prop:cal-Mor-Camp-est}
Let $v$ be a caloric function in $Q_r(z_0)$. Then there is $C=C(n)>0$ such that for any $0<\rho<r$,
\begin{align*}
    \int_{Q_\rho(z_0)}|\D v|^2&\le C\left(\frac\rho r\right)^{n+2}\int_{Q_r(z_0)}|\D v|^2,\\
    \int_{Q_\rho(z_0)}|v-\mean{v}_{z_0,\rho}|^2&\le C\left(\frac\rho r\right)^{n+4}\int_{Q_r(z_0)}|v-\mean{v}_{z_0,r}|^2,\\
    \int_{Q_\rho(z_0)}|\D v-\mean{\D v}_{z_0,\rho}|^2&\le C\left(\frac \rho r\right)^{n+4}\int_{Q_r(z_0)}|\D v-\mean{\D v}_{z_0,r}|^2.
\end{align*}
Moreover, for any $z_1\in Q_\rho(z_0)$ and $z_2\in Q_r(z_0)$, $$
\int_{Q_\rho(z_0)}|v-v(z_1)|^2\le C(n)\left(\frac\rho r\right)^{n+4}\int_{Q_r(z_0)}|v-v(z_2)|^2,\quad 0<\rho<r/2.
$$
\end{proposition}

\begin{proof}
By \cite{Cam66}*{Lemma~5.1}, we have that for any $z_1\in Q_\rho(z_0)$ and $0<\rho<r/2$, \begin{align*}
    \int_{Q_\rho(z_0)}v^2&\le C(n)\left(\frac\rho r\right)^{n+2}\int_{Q_r(z_0)}v^2,\\
    \int_{Q_\rho(z_0)}|v-v(z_1)|^2&\le C(n)\left(\frac\rho r\right)^{n+4}\int_{Q_r(z_0)}v^2.
\end{align*}
From the facts that $\D v$ and $v-C$ are caloric for any constant $C$ and $\mean{v}_{z_0,\rho}=v(z_1)$ for some $z_1\in Q_\rho(z_0)$, we immediately obtain the proposition for $0<\rho<r/2$. The first three inequalities in the proposition also holds for $0<\rho<r$, since the integrals there are monotone in the radius of the cylinders.
\end{proof}

As explained before, the Morrey-type estimates do not provide the H\"{o}lder continuity in the parabolic setting. Thus, to prove the H\"{o}lder regularity of almost caloric functions, we will derive their growth estimates with the functional
$$\vp_{z_0}(r,u)=\int_{Q_r(z_0)}r^{n+4}|\D u|^2+\int_{Q_r(z_0)\times Q_r(z_0)}|u(z)-u(w)|^2\,dzdw$$
and apply the Campanato-space embedding. For this purpose, we first prove the growth estimates for caloric functions.

\begin{proposition}\label{prop:cal-phi-est}
Let $v$ be a caloric function in $Q_r(z_0)$. Then for $0<\rho<r$, \begin{align}\label{eq:cal-Mor-est}
\vp_{z_0}(\rho,v)\le C(n)\left(\frac\rho r\right)^{2n+6}\vp_{z_0}(r,v).
\end{align}
\end{proposition}

\begin{proof}
Since $s\mapsto \vp_{z_0}(s, v)$  is nondecreasing on $s\in(0,r)$, \eqref{eq:cal-Mor-est} can be obtained easily for $r/2\le \rho<r$: $$
\vp_{z_0}(\rho,v)\le\left(\frac{2\rho}r\right)^{2n+6}\vp_{z_0}(r,v)=2^{2n+6}\left(\frac\rho r\right)^{2n+6}\vp_{z_0}(r,v).
$$
Thus we may assume $\rho<r/2$. Then Proposition~\ref{prop:cal-Mor-Camp-est} gives that for any $z_1\in Q_\rho(z_0)$ and $z_2\in Q_r(z_0)$
$$
\int_{Q_\rho(z_0)}(\rho^2|\D v|^2+|v-v(z_1)|^2)\le C(n)\left(\frac\rho r\right)^{n+4}\int_{Q_r(z_0)}(r^2|\D v|^2+|v-v(z_2)|^2).
$$
Taking averages with respect to $z_1$ over $Q_\rho(z_0)$ and with respect to $z_2$ over $Q_r(z_0)$, we get \begin{align*}
    &\int_{Q_\rho(z_0)}\rho^2|\D v|^2+\frac1{|Q_\rho|}\int_{Q_\rho(z_0)\times Q_\rho(z_0)}|v(z)-v(w)|^2\,dzdw\\
    &\qquad \le C(n)\left(\frac\rho r\right)^{n+4}\left(\int_{Q_r(z_0)}r^2|\D v|^2+\frac1{|Q_r|}\int_{Q_r(z_0)\times Q_r(z_0)}|v(z)-v(w)|^2\,dzdw\right),
\end{align*}
which implies \begin{align*}
    &\int_{Q_\rho(z_0)}\rho^2|\D v|^2+\frac1{\rho^{n+2}}\int_{Q_\rho(z_0)\times Q_\rho(z_0)}|v(z)-v(w)|^2\,dzdw\\
    &\qquad\le C(n)\left(\frac\rho r\right)^{n+4}\left(\int_{Q_r(z_0)}r^2|\D v|^2+\frac1{r^{n+2}}\int_{Q_r(z_0)\times Q_r(z_0)}|v(z)-v(w)|^2\,dzdw\right).
\end{align*}
Multiplying both sides by $\rho^{n+2}$, we conclude \eqref{eq:cal-Mor-est}.
\end{proof}

Next, we use the estimates for caloric functions (Propositions \ref{prop:cal-Mor-Camp-est} and \ref{prop:cal-phi-est}) to obtain the growth estimate for almost caloric functions. To use the result in the variable coefficient case, we state the following proposition for functions satisfying the almost caloric property at a point.

\begin{proposition}\label{prop:alm-cal-Mor-Camp-est} Suppose that $u$ satisfies the almost caloric property at $z_0$ in $Q_r(z_0)\Subset Q_1$. Then for any $0<\rho<r$, 
\begin{equation}
    \label{eq:alm-cal-Mor-est} \vp_{z_0}(\rho, u)\le C(n,\al)\left[\left(\frac\rho r\right)^{2n+6}+r^\al\right]\vp_{z_0}(r,u),
    \end{equation}
    \begin{multline}\label{eq:alm-cal-Camp-est}\int_{Q_\rho(z_0)}|\D u-\mean{\D u}_{z_0,\rho}|^2 \le C(n,\al)\left(\frac\rho r\right)^{n+4}\int_{Q_r(z_0)}|\D u-\mean{\D u}_{z_0,r}|^2\\+C(n)r^\al\int_{Q_r(z_0)}|\D u|^2.\end{multline}
\end{proposition}

\begin{proof}
Let $v$ be the caloric function in $Q_r(z_0)$ with $v=u$ on $\pa_pQ_r(z_0)$. Then, by the almost caloric property of $u$,
\begin{align*}
    \hspace{2em}&\hspace{-2em}\int_{Q_r(z_0)}|\D(u-v)|^2\\
    &=\int_{Q_r(z_0)}\left(|\D u|^2-|\D v|^2\right)+2\int_{Q_r(z_0)}\D v\D(v-u)\\
    &\le \left(2\int_{Q_r(z_0)}\pa_tu(v-u)+r^\al\int_{Q_r(z_0)}(|\D u|^2+|\D v|^2)\right)-2\int_{Q_r(z_0)}\pa_t v(v-u)\\
    &= 2\int_{Q_r(z_0)}\pa_t(u-v)(v-u)+r^\al\int_{Q_r(z_0)}(|\D u|^2+|\D v|^2)\\
    &=-\int_{Q_r(z_0)}\pa_t(|u-v|^2)+r^\al\int_{Q_r(z_0)}(|\D u|^2+|\D v|^2)\\
    &\le r^\al\int_{Q_r(z_0)}(|\D u|^2+|\D v|^2).\end{align*}
Thus, 
\begin{align*}
    \int_{Q_r(z_0)}|\D v|^2&\le 2\int_{Q_r(z_0)}|\D u|^2+2\int_{Q_r(z_0)}|\D(u-v)|^2\\
    &\le 2\int_{Q_r(z_0)}|\D u|^2+2r^\al\int_{Q_r(z_0)}(|\D u|^2+|\D v|^2),
\end{align*}
and hence $$\int_{Q_r(z_0)}|\D v|^2\le 4\int_{Q_r(z_0)}|\D u|^2,\quad 0<r\le r_0,
$$
for some $r_0=r_0(\al)>0$ small.
Combining the above estimates, \begin{align}\label{eq:cal-alm-cal-grad-diff-est}
\int_{Q_r(z_0)}|\D(u-v)|^2\le r^\al\int_{Q_r(z_0)}(|\D u|^2+|\D v|^2)\le 5r^\al\int_{Q_r(z_0)}|\D u|^2.
\end{align}
Moreover, applying the Poincar\'e inequality to $u(\cdot,t)-v(\cdot,t)\in W_0^{1,2}(B_{r}(x_0))$ for each $t\in(t_0-r^2, t_0)$ and using \eqref{eq:cal-alm-cal-grad-diff-est}, we obtain
\begin{align}\label{eq:cal-alm-cal-poincare}
\begin{split}
    \int_{Q_{r}(z_0)}|u-v|^2 &= \int_{t_0-r^2}^{t_0}\int_{B_r(x_0)}|u(x,t)-v(x,t)|^2\,dxdt\\
    &\le \int_{t_0-r^2}^{t_0}Cr^2\int_{B_r(x_0)}|\D(u-v)|^2\,dxdt\\
    &\le Cr^2\int_{Q_r(z_0)}|\D(u-v)|^2\le Cr^{2+\al}\int_{Q_r(z_0)}|\D u|^2.
\end{split}\end{align}
In addition, we observe that
\begin{align*}
    \vp_{z_0}(\rho,u)&=\int_{Q_\rho(z_0)}\rho^{n+4}|\D u|^2+\int_{Q_\rho(z_0)\times Q_\rho(z_0)}|u(z)-u(w)|^2\,dzdw\\
    &\begin{multlined}\le 2\int_{Q_\rho(z_0)}\rho^{n+4}|\D v|^2+2\int_{Q_\rho(z_0)}\rho^{n+4}|\D(u-v)|^2\\
    \qquad +3\int_{Q_\rho(z_0)\times Q_\rho(z_0)}(|v(z)-v(w)|^2+|u(z)-v(z)|^2+|u(w)-v(w)|^2)\,dzdw\end{multlined}\\
    &\le 3\vp_{z_0}(\rho,v)+2\int_{Q_\rho(z_0)}\rho^{n+4}|\D(u-v)|^2+C(n)\rho^{n+2}\int_{Q_\rho(z_0)}|u-v|^2,
 \end{align*}   
and similarly    
\begin{align*}
        \vp_{z_0}(r,v)\le 3\vp_{z_0}(r,u)+2\int_{Q_r(z_0)}r^{n+4}|\D(u-v)|^2+C(n)r^{n+2}\int_{Q_r(z_0)}|u-v|^2.
\end{align*}
These inequalities, together with \eqref{eq:cal-Mor-est}, \eqref{eq:cal-alm-cal-grad-diff-est} and \eqref{eq:cal-alm-cal-poincare}, yield that for $0<\rho<r\le r_0$,    
    \begin{align*}
    \vp_{z_0}(\rho,u)&\begin{multlined}[t]\le C(n)\left(\frac\rho r\right)^{2n+6}\vp_{z_0}(r,v)+2\int_{Q_\rho(z_0)}\rho^{n+4}|\D(u-v)|^2\\
    +C(n)\rho^{n+2}\int_{Q_\rho(z_0)}|u-v|^2\end{multlined}\\*
    &\begin{multlined}\le C(n)\left(\frac\rho r\right)^{2n+6}\vp_{z_0}(r,u)+C(n)r^{n+4}\int_{Q_r(z_0)}|\D(u-v)|^2\\
    +C(n)r^{n+2}\int_{Q_r(z_0)}|u-v|^2\end{multlined}\\
    &\le C(n)\left(\frac\rho r\right)^{2n+6}\vp_{z_0}(r,u)+C(n)r^{n+4+\al}\int_{Q_r(z_0)}|\D u|^2\\
    &\le C(n)\left[\left(\frac\rho r\right)^{2n+6}+r^\al\right]\vp_{z_0}(r,u).
\end{align*}
This implies \eqref{eq:alm-cal-Mor-est} when $0<\rho<r\le r_0(\al)$. To obtain \eqref{eq:alm-cal-Mor-est} for any $0<\rho<r<1$, we further consider the cases when $$
0<\rho<r_0<r<1\quad\text{or}\quad r_0\le \rho<r<1.
$$
If $0<\rho<r_0<r<1$, then \begin{align*}
    \vp_{z_0}(\rho,u)&\le C(n)\left[\left(\frac\rho{r_0}\right)^{2n+6}+r_0^\al\right]\vp_{z_0}(r_0,u)\le C(n,\al)\left[\rho^{2n+6}+r_0^\al\right]\vp_{z_0}(r_0,u)\\
    &\le C(n,\al)\left[\left(\frac\rho r\right)^{2n+6}+r^\al\right]\vp_{z_0}(r,u).
\end{align*}
On the other hand, if $r_0\le \rho<r<1$, then it simply follows from $\rho/r\ge r_0$ that $$
\vp_{z_0}(\rho, u)\le \left(\frac1{r_0}\right)^{2n+6}\left(\frac\rho r\right)^{2n+6}\vp_{z_0}(r,u)\le C(n,\al)\left(\frac\rho r\right)^{2n+6}\vp_{z_0}(r,u).
$$

Next, we prove \eqref{eq:alm-cal-Camp-est}. Combining Proposition~\ref{prop:cal-Mor-Camp-est} with \eqref{eq:cal-alm-cal-grad-diff-est} gives that for $0<\rho<r\le r_0(\al)$,
\begin{align*}
    \hspace{2em}&\hspace{-2em}\int_{Q_\rho(z_0)}|\D u-\mean{\D u}_{z_0,\rho}|^2\\
    &\le  3\int_{Q_\rho(z_0)}|\D v-\mean{\D v}_{z_0,\rho}|^2+6\int_{Q_\rho(z_0)}|\D(u-v)|^2\\
    &\le C\left(\frac\rho r\right)^{n+4}\int_{Q_r(z_0)}|\D v-\mean{\D v}_{z_0,r}|^2+6\int_{Q_\rho(z_0)}|\D(u-v)|^2\\
    &\le C\left(\frac\rho r\right)^{n+4}\int_{Q_r(z_0)}|\D u-\mean{\D u}_{z_0,r}|^2+C\int_{Q_r(z_0)}|\D(u-v)|^2\\
    &\le C(n)\left(\frac\rho r\right)^{n+4}\int_{Q_r(z_0)}|\D u-\mean{\D u}_{z_0,r}|^2+C(n)r^\al\int_{Q_r(z_0)}|\D u|^2.
\end{align*}
To extend this inequality from $0<\rho<r\le r_0(\al)$ to $0<\rho<r<1$, we argue as above (by considering the cases when $0<\rho<r_0<r<1$ and $r_0\le\rho<r<1$). If $0<\rho<r_0<r<1$, then 
\begin{align*}
   \hspace{2em} &\hspace{-2em}\int_{Q_\rho(z_0)}|\D u-\mean{\D u}_{z_0,\rho}|^2\\
    &\le C(n)\left(\frac\rho{r_0}\right)^{n+4}\int_{Q_{r_0}(z_0)}|\D u-\mean{\D u}_{z_0,r_0}|^2+C(n)r_0^\al\int_{Q_{r_0}(z_0)}|\D u|^2\\
    &\le C(n,\al)\left(\frac\rho{r}\right)^{n+4}\int_{Q_{r}(z_0)}|\D u-\mean{\D u}_{z_0,r}|^2+C(n)r^\al\int_{Q_{r}(z_0)}|\D u|^2.
\end{align*}
Finally, if $r_0\le \rho<r<1$, then using $\rho/r\ge r_0$, we obtain \begin{align*}
    \int_{Q_\rho(z_0)}|\D u-\mean{\D u}_{z_0,\rho}|^2&\le \left(\frac1{r_0}\right)^{n+4}\left(\frac\rho r\right)^{n+4}\int_{Q_r(z_0)}|\D u-\mean{\D u}_{z_0,r}|^2\\
    &\le C(n,\al)\left(\frac\rho r\right)^{n+4}\int_{Q_r(z_0)}|\D u-\mean{\D u}_{z_0,r}|^2.\qedhere
\end{align*}

\end{proof}

The following is a useful lemma that we will use below, whose proof can be found e.g.\ in
\cite{HanLin97}.
\begin{lemma}\label{lem:HL}
  Let $r_0>0$ be a positive number and let
  $\phi:(0,r_0)\to (0, \infty)$ be a nondecreasing function. Let $a$,
  $\beta$, and $\gamma$ be such that $a>0$, $\gamma >\beta >0$. There
  exist two positive numbers $\e=\e(a,\gamma,\beta)$,
  $c=c(a,\gamma,\beta)$ such that, if
$$
\phi(\rho)\le
a\Bigl[\Bigl(\frac{\rho}{r}\Bigr)^{\gamma}+\e\Bigr]\phi(r)+b\, r^{\be}
$$ for all $\rho$, $r$ with $0<\rho\leq r<r_0$, where $b\ge 0$,
then one also has, still for $0<\rho<r<r_0$,
$$
\phi(\rho)\le
c\Bigl[\Bigl(\frac{\rho}{r}\Bigr)^{\be}\phi(r)+b\rho^{\be}\Bigr].
$$
\end{lemma}

The following is the parabolic version of Campanato space embedding, whose proof is similar to its elliptic version (e.g. \cite{HanLin97}*{Theorem~3.1}).

\begin{lemma}\label{lem:par-camp-space-embed}
Let $u\in L^2(Q_1)$ and $M$ be such that $\|u\|_{L^2(Q_1)}\le M$. Suppose that for some $\si\in (0,1)$ and $r_0>0$ $$
\int_{Q_r(z)}|u-\mean{u}_{z,r}|^2\le M^2r^{n+2+2\si}
$$
for any $Q_r(z)\Subset Q_1$ with $r<r_0$. Then $u\in H^{\si,\si/2}(Q_1)$, and for any $K\Subset Q_1$ $$
\|u\|_{H^{\si,\si/2}(K)}\le C(n,K,\si,r_0)M.
$$
\end{lemma}

\begin{theorem}\label{thm:alm-cal-reg}
Let $u$ be an almost caloric function in $Q_1$. Then   \begin{enumerate}
  \item $u\in H^{\si,\si/2}(Q_1)$ for all $\si\in(0,1)$;
  \item $\D u\in H^{\al/2, \al/4}(Q_1)$.
  \end{enumerate}
\end{theorem}

\begin{proof}
Let $K\Subset Q_1$, and fix $z_0\in K$ and $0<\si<1$. Let $r_0:=\frac12d_{\rm par}(K,\pa Q_1)>0$.
If $0<\rho<r<r_0$, then applying Lemma~\ref{lem:HL} to \eqref{eq:alm-cal-Mor-est} gives that for $0<\rho<r<r_0$, $$
\vp_{z_0}(\rho,u)\le C(n,\al,\si)\left(\frac\rho r\right)^{2n+4+2\si}\vp_{z_0}(r,u).
$$
Taking $r\nearrow r_0$ yields 
\begin{multline}\label{eq:alm-cal-phi-bound}
    \int_{Q_\rho(z_0)\times Q_\rho(z_0)}|u(z)-u(w)|^2\,dzdw+\rho^{n+4}\int_{Q_\rho(z_0)}|\D u|^2\\
    = \vp_{z_0}(\rho,u)\le C(n,\al,\si)\left(\frac\rho{r_0}\right)^{2n+4+2\si}\vp_{z_0}(r_0,u)\\
    \le C(n,K,\al,\si)\|u\|^2_{W^{1,0}_2(Q_1)}\rho^{2n+4+2\si}.
\end{multline}
Combining this with the observation that for each $w\in Q_\rho(z_0)$
\begin{align*}
    |\mean{u}_{z_0,\rho}-u(w)|^2&= \left|\frac1{|Q_\rho|}\int_{Q_\rho(z_0)}u(z)\,dz-\frac1{|Q_\rho|}\int_{Q_\rho(z_0)}u(w)\,dz\right|^2\\
    &\le \frac1{|Q_\rho|^2}\left(\int_{Q_\rho(z_0)}|u(z)-u(w)|\,dz\right)^2\\
    &\le \frac{C(n)}{\rho^{n+2}}\int_{Q_\rho(z_0)}|u(z)-u(w)|^2\,dz,
\end{align*}
we have \begin{align*}
    \int_{Q_\rho(z_0)}|\mean{u}_{z_0,\rho}-u(w)|^2\,dw&\le  \frac{C(n)}{\rho^{n+2}}\int_{Q_\rho(z_0)\times Q_\rho(z_0)}|u(z)-u(w)|^2\,dzdw\\
    &\le C(n,K,\al,\si)\|u\|^2_{W^{1,0}_2(Q_1)}\rho^{n+2+2\si}.
\end{align*}
By Lemma~\ref{lem:par-camp-space-embed}, this implies $u\in H^{\si,\si/2}$.\\
\indent For the regularity of $\D u$, let $\widetilde{K}:=\{z\in Q_1: d_{\rm par}(z,\pa_pQ_1)\ge 1/2d_{\rm par}(K,\pa_pQ_1)\}$. Note that \eqref{eq:alm-cal-phi-bound} gives \begin{align*}
    \int_{Q_\rho(z_0)}|\D u|^2\le C(n,K,\al,\si)\|u\|^2_{W^{1,0}_2(Q_1)}\rho^{n+2\si},\quad 0<\rho<r_0.
\end{align*}
We then combine this estimate with $\si=1-\al/4$ and \eqref{eq:alm-cal-Camp-est} to obtain 
\begin{equation*}
    \int_{Q_\rho(z_0)}|\D u-\mean{\D u}_{z_0,\rho}|^2\le C\left(\frac\rho r\right)^{n+4}\int_{Q_r(z_0)}|\D u-\mean{\D u}_{z_0,r}|^2
    +C\|u\|^2_{W^{1,0}_2(Q_1)}r^{n+2+\al/2}
\end{equation*}
for $z_0\in\widetilde{K}$ and $0<\rho<r<r_0$. Taking $r\nearrow r_0$ and applying Lemma~\ref{lem:par-camp-space-embed}, we see that $\D u\in H^{\al/4,\al/8}$ with $\|\D u\|_{H^{\al/4,\al/8}(\widetilde{K})}\le C(n,K,\al)\|u\|_{W^{1,0}_2(Q_1)}$. Note that the definition $r_0=\frac12d_{\rm par}(K,\pa Q_1)$ implies $Q_{r_0}(z_0)\subset \widetilde{K}$ for all $z_0\in K$. Now we combine \eqref{eq:alm-cal-Camp-est} with the improved estimate $\int_{Q_r(z_0)}|\D u|^2\le C\|u\|_{W^{1,0}_2(Q_1)}^2r^{n+2}$ to obtain 
\begin{equation}
    \label{eq:alm-cal-Camp-est-improved}
     \int_{Q_\rho(z_0)}|\D u-\mean{\D u}_{z_0,\rho}|^2\le C\left(\frac\rho r\right)^{n+4}\int_{Q_r(z_0)}|\D u-\mean{\D u}_{z_0,r}|^2
     +C\|u\|_{W^{1,0}_2(Q_1)}^2r^{n+2+\al},
\end{equation}
for $z_0\in K$ and $0<\rho<r<r_0$.
This implies $\D u\in H^{\al/2,\al/4}$.
\end{proof}

%%%%%%%%%%%%%%%%%%%%%%%%%%%%%%%%%%%%%%%%%%%%

\section{H\"older regularity of almost minimizers}\label{sec:alm-min-holder}

This section is devoted to proving $H^{\si,\si/2}$-regularity of almost minimizers, $0<\si<1$. As in Section~\ref{sec:alm-cal} concerning almost caloric functions, we first prove growth estimates for solutions of the parabolic Signorini problem (Proposition~\ref{prop:par-sig-phi-est}), and then obtain the similar result for almost minimizers (Proposition~\ref{prop:par-alm-min-Mor-est-thin}).

Before starting proving Proposition~\ref{prop:par-sig-phi-est}, we observe that the combination of Theorem~9.1 in \cite{DanGarPetTo17} and Theorem 1 in \cite{BanDanGarPet20} implies that if $v$ is a solution of the parabolic Signorini problem in $Q_1$, then $v\in H^{3/2,3/4}(K^\pm\cup K')$ for any $K\Subset Q_1$ and $$
\|v\|_{H^{3/2,3/4}(K^\pm\cup K')}\le C(n,K)\|v\|_{L^2(Q_1)}.
$$
We also define, for any $Q_r(z_0)\subset Q_1$ with $z_0\in Q'_1$, $$
a_{v,z_0,r}=\begin{cases}
  0,&\text{if } \{v=0\}\cap Q'_r(z_0)=\Lambda(v)\cap Q'_r(z_0)\neq\emptyset,\\
  \mean{v}_{z_0,r},& \text{if } \Lambda(v)\cap Q'_r(z_0)=\emptyset.
\end{cases}
$$

\begin{proposition}
Let $v$ be a solution of the parabolic Signorini problem in $Q_r(z_0)\Subset Q_1$ with $z_0\in Q'_1$. Then, for any $0<\rho<r$, $$
\int_{Q_\rho(z_0)}|v-a_{v,z_0,\rho}|^2\le C(n)\left(\frac\rho r\right)^{n+4}\int_{Q_r(z_0)}|v-a_{v,z_0,r}|^2.
$$
\end{proposition}

\begin{proof}
Notice that if $a_{v,z_0,r}=\mean{v}_{z_0,r}$ then $a_{v,z_0,\rho}=\mean{v}_{z_0,\rho}$ as well. Thus we have three possibilities $(a_{v,z_0,\rho},a_{v,z_0,r})=(0,0)\, , (\mean{v}_{z_0,\rho},0)\, ,(\mean{v}_{z_0,\rho},\mean{v}_{z_0,r})$, and we have in all of these cases $$
\int_{Q_\rho(z_0)}|v-a_{v,z_0,\rho}|^2\le \int_{Q_r(z_0)}|v-a_{v,z_0,r}|^2.
$$
Thus, if $\rho\ge r/2$, then we simply have $$
\int_{Q_\rho(z_0)}|v-a_{v,z_0,\rho}|^2\le 2^{n+4}\left(\frac\rho r\right)^{n+4} \int_{Q_r(z_0)}|v-a_{v,z_0,r}|^2.
$$
Therefore, we may assume that $\rho<r/2$.

\medskip\noindent\emph{Case 1.} If  $a_{v,z_0,r}=\mean{v}_{z_0,r}$, then $a_{v,z_0,\rho}=\mean{v}_{z_0,\rho}$ and $v$ is caloric in $Q_r(z_0)$. Thus, by Proposition~\ref{prop:cal-Mor-Camp-est} $$
\int_{Q_\rho(z_0)}|v-\mean{v}_{z_0,\rho}|^2\le C(n)\left(\frac\rho r\right)^{n+4}\int_{Q_r(z_0)}|v-\mean{v}_{z_0,r}|^2.
$$
\medskip\noindent\emph{Case 2.} Suppose $a_{v,z_0,r}=0$. Using that $a_{v,z_0,\rho}=v(z_1)$ for some $z_1\in Q_\rho(z_0)$ and applying $H^{3/2,3/4}$-estimate of $v$ in $Q^\pm_{r/2}(z_0)$, we obtain
\begin{align*}
    \int_{Q_\rho(z_0)}|v-a_{v,z_0,\rho}|^2 &\le \int_{Q_\rho(z_0)}2\left(|v-v(z_0)|^2+|v(z_0)-v(z_1)|^2\right)\\
    &\le C(n)\rho^{n+2}\left(\|\D v\|^2_{L^\infty(Q_{r/2}(z_0))}\rho^2+\|\pa_tv\|^2_{L^\infty(Q_{r/2}(z_0))}\rho^4\right)\\
    &\le C(n)\rho^{n+2}\left(\frac{\rho^2}{r^{n+4}}\|v\|^2_{L^2(Q_r(z_0))}+\frac{\rho^4}{r^{n+6}}\|v\|^2_{L^2(Q_r(z_0))}\right)\\
    &\le C(n)\left(\frac\rho r\right)^{n+4}\int_{Q_r(z_0)}v^2= C(n)\left(\frac\rho r\right)^{n+4}\int_{Q_r(z_0)}|v-a_{v,z_0,r}|^2.
\end{align*}
This completes the proof.
\end{proof}

\begin{proposition}
Let $v$ be a solution of the parabolic Signorini problem in $Q_r(z_0)\Subset Q_1$ with $z_0\in Q'_1$. Then, for any $0<\rho<r$, $$
\int_{Q_\rho(z_0)}|\D v|^2\le C(n)\left(\frac\rho r\right)^{n+2}\int_{Q_r(z_0)}|\D v|^2+C(n)\frac{\rho^{n+2}
}{r^{n+4}}\int_{Q_r(z_0)}|v-a_{v,z_0,r}|^2.
$$
\end{proposition}

\begin{proof}
If $\rho\ge r/2$, then we simply have $$
\int_{Q_\rho(z_0)}|\D v|^2\le 2^{n+2}\left(\frac\rho r\right)^{n+2}\int_{Q_r(z_0)}|\D v|^2.
$$
On the other hand, when $\rho<r/2$, we have $$
\int_{Q_\rho(z_0)}|\D v|^2\le C(n)\rho^{n+2}\|\D v\|^2_{L^\infty(Q_{r/2}(z_0))}\le C(n)\frac{\rho^{n+2}}{r^{n+4}}\int_{Q_r(z_0)}v^2.
$$
Clearly, if $a_{v,z_0,r}=0$, then we are done. Thus, we may assume $a_{v,z_0,r}=\mean{v}_{z_0,r}$, which implies that $v$ is caloric in $Q_r(z_0)$. Then we can replace $v$ by $v-a_{v,z_0,r}$ in the above inequality, and get $$
\int_{Q_\rho(z_0)}|\D v|^2\le C(n)\frac{\rho^{n+2}}{r^{n+4}}\int_{Q_r(z_0)}|v-a_{v,z_0,r}|^2.
$$
This completes the proof.
\end{proof}

The combination of the preceding two propositions gives the following inequality.

\begin{corollary}
Let $v$ be a solution of the parabolic Signorini problem in $Q_r(z_0)\Subset Q_1$ with $z_0\in Q'_1$. Then for any $0<\rho<r$, \begin{equation}\label{eq:par-sig-Mor-est-1}
\int_{Q_\rho(z_0)}(\rho^2|\D v|^2+|v-a_{v,z_0,\rho}|^2)\le C(n)\left(\frac\rho r\right)^{n+4}\int_{Q_r(z_0)}(r^2|\D v|^2+|v-a_{v,z_0,r}|^2).
\end{equation}
\end{corollary}

\begin{proposition}\label{prop:par-sig-phi-est}
Let $v$ be a solution of the parabolic Signorini problem in $Q_R(z_0)\Subset Q_1$ with $z_0\in Q'_1$. Then for $0<\rho<R$, \begin{align}\label{eq:par-sig-Mor-est-2}
\vp_{z_0}(\rho,v)\le C(n)\left(\frac\rho R\right)^{2n+6}\vp_{z_0}(R,v).
\end{align}
\end{proposition}

\begin{proof}
We first claim that there exists a constant $k=k(n)>2$ such that for any $0<\rho<\frac r2<\frac R{2k}$, $z_1\in Q_\rho(z_0)$ and $z_2\in Q_r(z_0)$, 
\begin{equation}\label{eq:par-sig-Mor-est-3}
\int_{Q_\rho(z_0)}(\rho^2|\D v|^2+|v-v(z_1)|^2)\\
\le C(n)\left(\frac\rho {kr}\right)^{n+4}\int_{Q_{kr}(z_0)}((kr)^2|\D v|^2+|v-v(z_2)|^2).
\end{equation}
Indeed, if $a_{v,z_0,r}=\mean{v}_{z_0,r}$, then $v$ is caloric in $Q_r(z_0)$, thus \eqref{eq:par-sig-Mor-est-3} readily follows from Proposition~\ref{prop:cal-Mor-Camp-est}. Hence we may assume $a_{v,z_0,r}=0$. We can take $z_\rho\in Q_\rho(z_0)$ and $z_r\in Q'_r(z_0)$ such that $v(z_\rho)=a_{v,z_0,\rho}$ and $v(z_r)=a_{v,z_0,r}=0$. Since $z_1$, $z_\rho\in Q_\rho(z_0)$, we have by $H^{3/2,3/4}$-estimate of $v$ in $Q^\pm_{r/2}(z_0)$ that for $\rho<r/2$ 
\begin{align*}
    |v(z_\rho)-v(z_1)|&\le |v(z_\rho)-v(z_0)|+|v(z_0)-v(z_1)|\\
    &\le 2\left(\|\D v\|_{L^\infty(Q_{r/2}(z_0))}\rho+\|\pa_tv\|_{L^\infty(Q_{r/2}(z_0))}\rho^2\right)\\
    &\le C(n)\frac\rho{r^{\frac n2+2}}\|v\|_{L^2(Q_r(z_0))},
\end{align*}
and thus 
\begin{align}\label{eq:par-sig-diff-est}
    \int_{Q_\rho(z_0)}|v-v(z_1)|^2&\le 2\int_{Q_\rho(z_0)}|v-v(z_\rho)|^2+2\int_{Q_\rho(z_0)}|v(z_\rho)-v(z_1)|^2\\
    &\le 2\int_{Q_\rho(z_0)}|v-a_{v,z_0,\rho}|^2+C(n)\left(\frac\rho r\right)^{n+4}\int_{Q_r(z_0)}v^2\nonumber.
\end{align}
Moreover, for $k=k(n)>2$ to be determined below, we deduce from $z_2\in Q_r(z_0)$ and $z_r\in Q'_r(z_0)$ that
\begin{align*}
    2\int_{Q_{kr}(z_0)}v(z_2)^2&= C(n)(kr)^{n+2}|v(z_2)-v(z_r)|^2\\
    &\le C(n)(kr)^{n+2}(\|\D v\|_{L^\infty(Q_{\frac k2r}(z_0))}^2r^2+\|\pa_tv\|^2_{L^\infty(Q_{\frac k2r}(z_0))}r^4)\\
    &\le C_0(n)(kr)^{n+2}\|v\|^2_{L^2(Q_{kr}(z_0))}
\frac{r^2}{(kr)^{n+4}}\\
&=\frac{C_0(n)}{k^2}\int_{Q_{kr}(z_0)}v^2.
\end{align*}
Take $k=k(n)>2$ large so that $\frac{C_0(n)}{k^2}\le 1/2$. Then \begin{align*}
    \int_{Q_{kr}(z_0)}v^2&\le 2\int_{Q_{kr}(z_0)}|v-v(z_2)|^2+2\int_{Q_{kr}(z_0)}v(z_2)^2\\
    &\le 2\int_{Q_{kr}(z_0)}|v-v(z_2)|^2+1/2\int_{Q_{kr}(z_0)}v^2,
\end{align*}
and thus 
$$
\int_{Q_{kr}(z_0)}v^2\le 4\int_{Q_{kr}(z_0)}|v-v(z_2)|^2.
$$
Combining this inequality with \eqref{eq:par-sig-Mor-est-1} and \eqref{eq:par-sig-diff-est}, we can obtain the claim \eqref{eq:par-sig-Mor-est-3}: \begin{align*}
    \hspace{2em}&\hspace{-2em}\int_{Q_\rho(z_0)}(\rho^2|\D v|^2+|v-v(z_1)|^2)\\
    &\le 2\int_{Q_\rho(z_0)}(\rho^2|\D v|^2+|v-a_{v,z_0,\rho}|^2)+C(n)\left(\frac\rho r\right)^{n+4}\int_{Q_r(z_0)}v^2\\
    &\le C(n)\left(\frac\rho r\right)^{n+4}\int_{Q_r(z_0)}(r^2|\D v|^2+v^2)+C(n)\left(\frac\rho r\right)^{n+4}\int_{Q_r(z_0)}v^2\\
    &\le C(n)\left(\frac\rho{kr}\right)^{n+4}\int_{Q_{kr}(z_0)}((kr)^2|\D v|^2+v^2)\\
    &\le C(n)\left(\frac\rho{kr}\right)^{n+4}\int_{Q_{kr}(z_0)}((kr)^2|\D v|^2+|v-v(z_2)|^2).
\end{align*}

\medskip

Now we take averages of \eqref{eq:par-sig-Mor-est-3} with respect to $z_1$ over $Q_\rho(z_0)$ and with respect to $z_2$ over $Q_r(z_0)$ to have for $\rho<r/2<\frac{R}{2k}$ \begin{align*}
    \hspace{2em}&\hspace{-2em}\int_{Q_\rho(z_0)}\rho^2|\D v|^2+\frac1{\rho^{n+2}}\int_{Q_\rho(z_0)\times Q_\rho(z_0)}|v(z)-v(w)|^2\,dzdw\\*
    &\le C(n)\left(\int_{Q_\rho(z_0)}\rho^2|\D v|^2+\frac1{|Q_\rho|}\int_{Q_\rho(z_0)\times Q_\rho(z_0)}|v(z)-v(w)|^2\,dzdw\right)\\
    &\le C(n)\left(\frac\rho{kr}\right)^{n+4}\left(\int_{Q_{kr}(z_0)}(kr)^2|\D v|^2+\frac1{|Q_r|}\int_{Q_r(z_0)\times Q_{kr}(z_0)}|v(z)-v(w)|^2\,dzdw\right)\\
    &\le C(n)\left(\frac\rho{kr}\right)^{n+4}\left(\int_{Q_{kr}(z_0)}(kr)^2|\D v|^2+\frac1{(kr)^{n+2}}\int_{Q_{kr}(z_0)\times Q_{kr}(z_0)}|v(z)-v(w)|^2\,dzdw\right).
\end{align*}
Multiplying $\rho^{n+2}$ and taking $r\nearrow R/k$, we get $$
\vp_{z_0}(\rho,v)\le C(n)\left(\frac\rho R\right)^{2n+6}\vp_{z_0}(R,v),\quad 0<\rho<\frac R{2k(n)}.
$$
This inequality also holds when $\frac{R}{2k}\le\rho<R$, for $s\mapsto \vp_{z_0}(s,v)$ is nondecreasing. This completes the proof.
\end{proof}

In the remaining of this section, we extend the above result for the solution of the parabolic Signorini problem to almost minimizers.

\begin{definition}[Almost parabolic Signorini property at a point]
We say that $u\in W^{1,1}_2(Q_1)$ satisfies the \emph{almost parabolic Signorini property} at $z_0=(x_0,t_0)$ in $Q_1$, if $u\ge 0$ on $Q_1'$ and $$
\int_{Q_r(z_0)}(1-r^\al)|\D u|^2+2\pa_tu(u-v)\le (1+r^\al)\int_{Q_r(z_0)}|\D v|^2
$$
for any $Q_r(z_0)\Subset Q_1$ and for any $v\in W^{1,0}_2(Q_r(z_0))$ with $v\ge 0$ on $Q'_r(z_0)$ and $v-u\in L^2(t_0-r^2,t_0;W^{1,2}_0(B_r(x_0)))$. 
\end{definition}

Notice that $u$ is an almost minimizer for the parabolic Signorini problem in $Q_1$ if and only if it satisfies the almost parabolic Signorini property at every $z_0\in Q_1$.

Again, in order to use the growth estimates obtained in this section when we deal with variable coefficient case in Section~\ref{sec:alm-min-var}, we state the following proposition with functions satisfying the almost parabolic Signorini property at a point.

\begin{proposition}
\label{prop:par-alm-min-Mor-est-thin}
Suppose that $u$ satisfies the almost parabolic Signorini property at $z_0$ in $Q_r(z_0)\Subset Q_1$ with $z_0\in Q'_1$. Then for any $0<\rho<r$, \begin{align}
    \label{eq:par-alm-min-Mor-est-thin}
    \vp_{z_0}(\rho,u)\le C(n,\al)\left[\left(\frac\rho r\right)^{2n+6}+r^\al\right]\vp_{z_0}(r,u).
\end{align}
\end{proposition}

\begin{proof}
Let $v$ be the parabolic Signorini replacement of $u$ in $Q_r(z_0)$ (i.e. $v$ is the solution of the parabolic Signorini problem in $Q_r(z_0)$ with $v=u$ on $\pa_pQ_r(z_0)$). Then, by the almost Signorini property of $u$ and the variational inequality of $v$,
\begin{align*}
    \hspace{2em}&\hspace{-2em}\int_{Q_r(z_0)}|\D(u-v)|^2\\
    &=\int_{Q_r(z_0)}\left(|\D u|^2-|\D v|^2\right)+2\int_{Q_r(z_0)}\D v\D(v-u)\\
    &\le \left(2\int_{Q_r(z_0)}\pa_tu(v-u)+r^\al\int_{Q_r(z_0)}(|\D u|^2+|\D v|^2)\right)-2\int_{Q_r(z_0)}\pa_t v(v-u)
    \\
    &= 2\int_{Q_r(z_0)}\pa_t(u-v)(v-u)+r^\al\int_{Q_r(z_0)}(|\D u|^2+|\D v|^2)\\
    &\le r^\al\int_{Q_r(z_0)}(|\D u|^2+|\D v|^2).
\end{align*}
As we have seen in the proof of Proposition~\ref{prop:alm-cal-Mor-Camp-est}, this gives that for $r<r_0(\al)$ 
  \begin{align}\label{eq:par-sig-alm-min-grad-est}
\int_{Q_r(z_0)}|\D v|^2&\le 4\int_{Q_r(z_0)}|\D u|^2,\\
\label{eq:par-sig-alm-min-grad-diff-est}
\int_{Q_r(z_0)}|\D(u-v)|^2&\le C(n)r^{\al}\int_{Q_r(z_0)}|\D u|^2,\\
\int_{Q_r(z_0)}|u-v|^2&\le C(n)r^{2+\al}\int_{Q_r(z_0)}|\D u|^2.
\end{align}
With these inequalities and \eqref{eq:par-sig-Mor-est-2} at hand, we can follow Proposition~\ref{prop:alm-cal-Mor-Camp-est} again to conclude \eqref{eq:par-alm-min-Mor-est-thin}.
\end{proof}

Now we are ready to prove the first part of Theorem~\ref{mthm:I}.

\begin{theorem}\label{thm:par-alm-min-holder}
Let $u$ be an almost minimizer of the parabolic Signorini problem in $Q_1$. Then $u\in H^{\si,\si/2}(Q_1)$ for every $\si\in(0,1)$. Moreover, for any $K\Subset Q_1$, \begin{align}\label{eq:par-alm-min-holder-est}
\|u\|_{H^{\si,\si/2}(K)}\le C(K,n,\al,\si)\|u\|_{W^{1,0}_2(Q_1)}.
\end{align}
\end{theorem}

\begin{proof} 
We may assume that $K\Subset Q_1$ is a parabolic cylinder centered on $Q'_1$. We claim that for any $z_0\in K$ and $0<\rho<r<r_0=r_0(n,K,\al,\si)$, \begin{align}
    \label{eq:par-alm-min-Mor-est}
    \vp_{z_0}(\rho,u)\le C(n,\al)\left[\left(\frac\rho r\right)^{2n+6}+r^\al\right]\vp_{z_0}(r,u),
\end{align}
which readily implies that $u\in H^{\si,\si/2}$, see the proof of Theorem~\ref{thm:alm-cal-reg}.

Indeed, if $r_0(n,K,\al,\si)$ is small enough, then $Q_{r_0}(z_0)\Subset Q_1$ for all $z_0\in K$. If $z_0\in K\cap Q'_1$, then \eqref{eq:par-alm-min-Mor-est} follows from Proposition~\ref{prop:par-alm-min-Mor-est-thin}, thus we may assume that $z_0=(x_0,t_0)\in K\cap Q^+_1$. In addition, if $\rho\ge r/4$, then we simply have
$$
\vp_{z_0}(\rho,u) \le
4^{2n+6}\left(\frac{\rho}{r}\right)^{2n+6}\vp_{z_0}(r,u),$$ hence we may also assume $\rho<r/4$. Let $d:=d_{\rm par}(z_0,Q_1')=\dist(x_0,B_1')>0$ and choose $x_1\in \pa B_d(x_0)\cap B_1'$. Then $z_1:=(x_1,t_0)\in\pa_{p}Q_d(z_0)\cap Q_1'$. From the assumption that $K$ is a cylinder centered on $Q_1'$, we have $z_1\in K\cap Q_1'$.

\medskip\noindent \emph{Case 1.} If $\rho\ge d$, then we use
$Q_{\rho}(z_0)\subset Q_{2\rho}(z_1)\subset Q_{r/2}(z_1)\subset
Q_r(z_0)$ and Proposition~\ref{prop:par-alm-min-Mor-est-thin} to write
\begin{align*}
  \vp_{z_0}(\rho,u)
  &\le \vp_{z_1}(2\rho,u)
    \le C\left[\left(\frac{2\rho}{r/2}\right)^{2n+6} +
    (r/2)^{\al}\right]\vp_{z_1}(r/2,u)\\
  &\le
    C\left[\left(\frac{\rho}{r}\right)^{2n+6}+r^{\al}\right]\vp_{z_0}(r,u).
\end{align*}

\medskip\noindent \emph{Case 2.}  Suppose now $d>\rho$. If $d>r$,
then $Q_r(z_0)\Subset Q_1^+$. Since $u$ is almost caloric in $Q_1^+$,
we can apply Proposition~\ref{prop:alm-cal-Mor-Camp-est} to have
$$
\vp_{z_0}(\rho, u)\le C(n,\al)\left[\left(\frac\rho r\right)^{2n+6}+r^\al\right]\vp_{z_0}(r,u).
$$
Thus, we may assume $d\le r$. We note that
$Q_d(z_0)\subset Q_1^+$ and, using a limiting argument from the previous
estimate, obtain
$$
\vp_{z_0}(\rho, u)\le C(n,\al)\left[\left(\frac\rho d\right)^{2n+6}+d^\al\right]\vp_{z_0}(d,u).
$$

\medskip\noindent \emph{Case 2.1.} If $r/4\le d$, then
$$
\vp_{z_0}(d,u) \le
4^{2n+6}\left(\frac{d}{r}\right)^{2n+6}\vp_{z_0}(r,u),
$$ 
which combined with the previous inequality implies \eqref{eq:par-alm-min-Mor-est}.

\medskip\noindent \emph{Case 2.2.}  It remains to consider the case
$\rho<d<r/4$. Using Proposition~\ref{prop:par-alm-min-Mor-est-thin} again, we conclude that
\begin{align*}
  \vp_{z_0}(d,u)
  &\le \vp_{z_1}(2d,u) \le C\left[\left(\frac{2d}{r/2}\right)^{2n+6}+(r/2)^{\al}\right]\vp_{z_1}(r/2,u) \\
  &\le C\left[\left(\frac{d}{r}\right)^{2n+6}+r^{\al}\right]\vp_{z_0}(r,u),
\end{align*}
which also implies \eqref{eq:par-alm-min-Mor-est}. This completes the proof.
\end{proof}

%%%%%%%%%%%%%%%%%%%%%%%%%%%%%%%%%%%%%%%%%%%%%%%%

\section{Higher regularity of almost minimizers}\label{sec:alm-min-grad-holder}

In this section we establish the higher regularity of the almost minimizer $u$, specifically the $H^{\be,\be/2}$-regularity of its spatial derivative $\D u$. We make use of the results in the previous sections to get a proper version of concentric cylinder estimates \eqref{eq:par-alm-min-grad-Camp-est-thin}, following the idea in \cite{JeoPet21}.

For a function $w\in W^{1,1}_2(Q_1)$, we define
$$
\wDw(x',x_n,t):=
\begin{cases}\D w(x', x_n,t),& x_n\ge 0, \\
  \D w(x', -x_n,t), &x_n<0,
\end{cases}
$$
the even extension of $\D w$ from $Q_R^+(z_0)$ to $Q_R(z_0)$.

\begin{proposition}\label{prop:par-sig-Camp-est} Let $v$ be a solution of the
  parabolic Signorini problem in $Q_R(z_0)\subset Q_1$ with $z_0\in Q'_1$. Then for $0<\rho\le R/2$, $$
\int_{Q_\rho(z_0)}|\wDv-\mean{\wDv}_{z_0,\rho}|^2\le C(n)\frac{\rho^{n+3}}{R^{n+5}}\int_{Q_R(z_0)}v^2.
$$
\end{proposition}

\begin{proof}
The proposition follows from the even symmetry of $\wDv$ and $H^{1/2,1/4}$-estimate of $\D v$ in $\overline{Q_{R/2}^+(z_0)}$:
 \begin{align*}
    \int_{Q_\rho(z_0)}|\wDv-\mean{\wDv}_{z_0,\rho}|^2&= 2\int_{Q_\rho^+(z_0)}|\D v-\mean{\D v}_{Q_\rho^+(z_0)}|^2\\
    &\le C(n)\rho^{n+3}\|\D v\|^2_{H^{1/2,1/4}(Q_{R/2}(z_0))}\\
    &\le C(n)\frac{\rho^{n+3}}{R^{n+5}}\int_{Q_R(z_0)}v^2.\qedhere
\end{align*}
\end{proof}

\begin{proposition}\label{prop:par-alm-min-grad-Camp-est-thin}
Let $K\Subset Q_1$ and suppose $u\in W^{1,1}_2(Q_1)$ satisfies the almost parabolic Signorini property at $z_0\in K'=K\cap Q_1'$. Then there exists a constant $r_0=r_0(n,\al,K)>0$ such that 
\begin{multline}\label{eq:par-alm-min-grad-Camp-est-thin}
    \int_{Q_\rho(z_0)}|\wDu-\mean{\wDu}_{z_0,\rho}|^2\le C(n)\left(\frac\rho r\right)^{n+4}\int_{Q_r(z_0)}|\wDu-\mean{\wDu}_{z_0,r}|^2\\ +C(n,\al,K)\|u\|^2_{W^{1,0}_2(Q_1)}r^{n+2+2\be},
\end{multline}
for $0<\rho<r<r_0$ with $\be=\frac\al{4(2n+4+\al)}.$
\end{proposition}

\begin{proof}
Let $r\in(0,r_0)$ with $r_0=r_0(n,\al,K)$ to be chosen below. Let us also denote
$$
\al':=1-\frac{\al}{8(n+2)}\in (0, 1),\quad R:=r^{\frac{2n+4}{2n+4+\al}}.
$$
Note that if $r_0=r_0(n,\al,K)$ is small, then $
\widetilde{K}:=\{z\in Q_1: d_{\rm par}(z,K)\le r_0^{\frac{2n+4}{2n+4+\al}}\}\Subset Q_1$ and $r\le R/2$. For notational simplicity, we assume $z_0=0$. We then split our proof into two cases:
$$
\sup_{\pa_p Q_{R}}|u|\le C_3 R ^{\al'}\quad\text{and}\quad\sup_{\pa_p
  Q_{R}}|u|> C_3R ^{\al'},$$ with
$C_3=2[u]_{H^{\al',\al'/2}(\K)}$.

\medskip\noindent \emph{Case 1.}  Assume first that
$\sup_{\pa_p Q_{R}}|u|\le C_3 R ^{\al'}$. Let $v$ be the parabolic Signorini
replacement of $u$ in $Q_{R}$. Then, Proposition~\ref{prop:par-sig-Camp-est} yields that for any $0<\rho<r$,
\begin{align*}
    \int_{Q_{\rho}}|\wDu-\mean{\wDu}_{\rho}|^2 &\le 3\int_{Q_{\rho}}|\wDv-\mean{\wDv}_{\rho}|^2 + 6\int_{Q_{\rho}}|\wDu-\wDv|^2 \\
     &\le C(n)\left(\int_{Q_{R}}v^2\right)\frac{r^{n+3}}{R^{n+5}}+6\int_{Q_{\rho}}|\wDu-\wDv|^2.
\end{align*}
Now, we estimate the two terms in the last line. For the first term, we observe that since $v^\pm=\max\{\pm v,0\}$ are subcaloric, $$
\sup_{Q_R}|v|=\sup_{\pa_pQ_R}|v|=\sup_{\pa_pQ_R}|u|\le C_3R^{\al'}.
$$
Combining this with the estimate $C_3\le C(n, \al, K)\|u\|_{W^{1, 0}_2(Q_1)}$, which follows from \eqref{eq:par-alm-min-holder-est}, we get
\begin{align*}
    \left(\int_{Q_{R}}v^2\right)\frac{r^{n+3}}{R^{n+5}}
    &\le C(n)\left(\sup_{Q_R}|v|\right)^2\frac{r^{n+3}}{R^{3}}\le C(n,\al,K)\|u\|_{W^{1,0}_2(Q_1)}^2r^{n+3}R^{2\al'-3}\\
    &\le C(n, \al, K)\|u\|_{W^{1,0}_2(Q_1)}^2r^{n+2+2\be}.
\end{align*}
For the second term $\int_{Q_{\rho}}|\wDu-\wDv|^2$, we apply Lemma~\ref{lem:HL} to \eqref{eq:par-alm-min-Mor-est-thin} to have for $0<R<S<R_0:=r_0^{\frac{2n+4}{2n+4+\al}}$
$$
\vp_0(R,u)\le C(n,\al)\left(\frac R S\right)^{2n+4+2\al'}\vp_0(S,u).
$$
Letting $S\nearrow R_0$ gives \begin{align}
    \label{eq:par-alm-min-grad-L2-bound}
    \int_{Q_R}|\D u|^2\le C(n,\al,K)\|u\|^2_{W^{1,0}_2(Q_1)}R^{n+2\al'}.
\end{align}
Equations \eqref{eq:par-sig-alm-min-grad-diff-est} and \eqref{eq:par-alm-min-grad-L2-bound} allow us to bound the second term
\begin{align}\label{eq:par-sig-alm-min-diff-Camp-est}
    \int_{Q_r}|\wDu-\wDv|^2 & \le \int_{Q_R}|\D(u-v)|^2  \le C(n)R^{\al}\int_{Q_{R}}|\D u|^2 \\
    &\nonumber \le  C(n, \al, K)\|u\|^2_{W^{1,0}_2(Q_1)}R^{n+\al+2\al'} \\
    &\nonumber \le C(n, \al, K)\|u\|^2_{W^{1,0}_2(Q_1)}r^{n+2+2\be}.
\end{align}
Therefore, we get 
$$
  \int_{Q_\rho}|\wDu-\mean{\wDu}_{\rho}|^2\le C(n,\al,K)\|u\|^2_{W^{1,0}_2(Q_1)}r^{n+2+2\be}.
$$

\medskip\noindent \emph{Case 2.}  Now suppose
$\sup_{\pa_p Q_{R}}|u|> C_3 R ^{\al'}$. By the choice of
$C_3=2[u]_{H^{\al',\al'/2}(\K)}$, we have either $u\ge (C_3/2) R^{\al'}$ in
all of $Q_{R}$ or $u\le -(C_3/2) R^{\al'}$ in all of $Q_{R}$. However,
from the inequality $u(0)\ge 0$, the only possibility is
$$
u\ge \frac{C_3}{2}\,R^{\al'}\quad \text{in } Q_{R}.
$$
Let $v$ again be the parabolic Signorini replacement of $u$ in $Q_R$. Then from
positivity of $v=u>0$ on $\pa_p Q_{R}$ and supercaloricity of $v$ in
$Q_{R}$, it follows that $v>0$ in $Q_{R}$ and is therefore caloric
there. Thus,
\begin{align}\label{eq:cal-grad-Camp-est}
\int_{Q_{\rho}}|\D v-\mean{\D v}_{\rho}|^2 \le
C(n)\Bigl(\frac{\rho}{r}\Bigr)^{n+4}\int_{Q_r}|\D v-\mean{\D v}_r|^2,\quad
0<\rho<r.
\end{align}
We next decompose $v$ into the sum of even and odd functions in $x_n$ \begin{align*}
    v(x',x_n,t)&=\frac{v(x',x_n,t)+v(x',-x_n,t)}2+\frac{v(x',x_n,t)-v(x',-x_n,t)}2\\
    &=:v^*(x',x_n,t)+v^\sharp(x',x_n,t).
\end{align*}
Then we have 
\begin{align*}
    \int_{Q_\rho}|\wDv-\mean{\wDv}_\rho|^2&\le 3\int_{Q_\rho}|\D v-\mean{\D v}_\rho|^2+6\int_{Q_\rho}|\wDv-\D v|^2\\
    &= 3\int_{Q_\rho}|\D v-\mean{\D v}_\rho|^2+6\int_{Q_\rho^-}(|2\pa_{x_n}v^*|^2+|2\D_{x'}v^\sharp|^2)\\
    &=3\int_{Q_\rho}|\D v-\mean{\D v}_\rho|^2+12\int_{Q_\rho}(|\pa_{x_n}v^*|^2+|\D_{x'}v^\sharp|^2).
\end{align*}
Similarly, we also have $$
\int_{Q_r}|\D v-\mean{\D v}_r|^2\le 3\int_{Q_r}|\wDv-\mean{\wDv}_r|^2+12\int_{Q_r}(|\pa_{x_n}v^*|^2+|\D_{x'}v^\sharp|^2).
$$
The above two inequalities, together with \eqref{eq:cal-grad-Camp-est}, yield that for all $0<\rho<r$
\begin{equation}
    \label{eq:cal-grad-sym-Camp-est}
    \int_{Q_\rho}|\wDv-\mean{\wDv}_\rho|^2\le C(n)\left(\frac\rho r\right)^{n+4}\int_{Q_r}|\wDv-\mean{\wDv}_r|^2+C(n)\int_{Q_r}(|\pa_{x_n}v^*|^2+|\D_{x'}v^\sharp|^2).
\end{equation}
Next, we estimate the last term in \eqref{eq:cal-grad-sym-Camp-est}. Note that since both $v$ and $v^\sharp$ are caloric in $Q_R$, $v^*$ is caloric in $Q_R$ as well. Combining this with \eqref{eq:par-sig-alm-min-grad-est} and \eqref{eq:par-alm-min-grad-L2-bound}, 
\begin{align*}
    \sup_{Q_{R/2}}|D^2_xv^*|^2+\sup_{Q_{R/2}}|D^2_xv^\sharp|^2&\le \frac{C(n)}{R^{n+4}}\int_{Q_R}(|\D v^*|^2+|\D v^\sharp|^2)\\
    &=\frac{C(n)}{R^{n+4}}\int_{Q_R}|\D v|^2\le \frac{C(n)}{R^{n+4}}\int_{Q_R}|\D u|^2\\
    &\le C(n,\al,K)\|u\|^2_{W^{1,0}_2(Q_1)}R^{2\al'-4}.
\end{align*}
Also, note that $\pa_{x_n}v^*=\D_{x'}v^\sharp=0$ on $Q'_{R/2}$. Thus, for $z=(x',x_n,t)\in Q_r$, we have 
\begin{align*}
    |\pa_{x_n}v^*(z)|^2+|\D_{x'}v^\sharp(z)|^2&\le |x_n|^2\left(\sup_{Q_{R/2}}|D^2_xv^*|^2+\sup_{Q_{R/2}}|D^2_xv^\sharp|^2\right)\\
    &\le C(n,\al,K)\|u\|^2_{W^{1,0}_2(Q_1)}r^2R^{2\al'-4}.
\end{align*}
Thus, it follows that 
\begin{align*}
    \int_{Q_r}(|\pa_{x_n}v^*|^2+|\D_{x'}v^\sharp|^2)&\le C(n,\al,K)\|u\|^2_{W^{1,0}_2(Q_1)}r^{n+4}R^{2\al'-4}\\
    &\le C(n,\al,K)\|u\|^2_{W^{1,0}_2(Q_1)}r^{n+2+2\be}.
\end{align*}
Combining this estimate and \eqref{eq:cal-grad-sym-Camp-est}, we obtain 
\begin{equation*}
    \int_{Q_\rho}|\wDv-\mean{\wDv}_\rho|^2\le C(n)\left(\frac\rho r\right)^{n+4}\int_{Q_r}|\wDv-\mean{\wDv}_r|^2+C(n,\al,K)\|u\|^2_{W^{1,0}_2(Q_1)}r^{n+2+2\be}.
\end{equation*}
Finally, by making use of this inequality and \eqref{eq:par-sig-alm-min-diff-Camp-est} (the equation \eqref{eq:par-sig-alm-min-diff-Camp-est} is valid in this case as well, for we have proved it in \emph{Case 1} without using the assumption $\sup_{\pa_pQ_R}|u|\le C_3R^{\al'}$), we conclude that
\begin{align*}
    \hspace{2em}&\hspace{-2em}\int_{Q_{\rho}}|\wDu-\mean{\wDu}_{\rho}|^2\\
    &\le 3\int_{Q_{\rho}}|\wDv-\mean{\wDv}_{\rho}|^2 + 6\int_{Q_{\rho}}|\wDu-\wDv|^2 \\
    &\le C(n)\Bigl(\frac{\rho}{r}\Bigr)^{n+4}\int_{Q_r}|\wDv-\mean{\wDv}_r|^2
    +C(n, \al, K)\|u\|_{W^{1,0}_2(Q_1)}^2r^{n+2+2\be}+6\int_{Q_{\rho}}|\wDu-\wDv|^2\\
    &\le C(n)\Bigl(\frac{\rho}{r}\Bigr)^{n+4}\int_{Q_r}|\wDu-\mean{\wDu}_r|^2
      +C(n, \al, K)\|u\|_{W^{1,0}_2(Q_1)}^2r^{n+2+2\be}+C\int_{Q_r}|\wDu-\wDv|^2\\
      &\le
      C(n)\Bigl(\frac{\rho}{r}\Bigr)^{n+4}\int_{Q_r}|\wDu-\mean{\wDu}_r|^2+ C(n, \al, K)\|u\|_{W^{1,0}_2(Q_1)}^2r^{n+2+2\be}.\qedhere
\end{align*}
\end{proof}

\begin{theorem}\label{thm:par-alm-min-grad-holder}
Let $u$ be an almost minimizer for the parabolic Signorini problem in $Q_1$. Then 
$$\wDu\in H^{\be,\be/2}(Q_1)\quad\text{with }\be=\frac\al{4(2n+4+\al)}.$$
Moreover, for any $K\Subset Q_1$ there holds \begin{align*}
    \|\wDu\|_{H^{\be,\be/2}(K)}\le C(n,\al,K)\|u\|_{W^{1,0}_2(Q_1)}.
\end{align*}
\end{theorem}

\begin{proof}
Let $K\Subset Q_1$ be a parabolic cylinder centered on $Q'_1$ and $r_0=r_0(n,\al,K)$ be as in Proposition~\ref{prop:par-alm-min-grad-Camp-est-thin}. We claim that for any $z_0\in K$ and $0<\rho<r<r_0$, 
\begin{multline}\label{eq:par-alm-min-grad-Camp-est}
    \int_{Q_\rho(z_0)}|\wDu-\mean{\wDu}_{z_0,\rho}|^2\le C(n,\al)\left(\frac\rho r\right)^{n+4}\int_{Q_r(z_0)}|\wDu-\mean{\wDu}_{z_0,r}|^2\\ +C(n,\al,K)\|u\|^2_{W^{1,0}_2(Q_1)}r^{n+2+2\be}.
\end{multline}
Indeed, when $z_0\in K\cap Q_1'$, \eqref{eq:par-alm-min-grad-Camp-est} simply follows from Proposition~\ref{prop:par-alm-min-grad-Camp-est-thin}. Thus we may assume that $z_0=(x_0,t_0)\in K\cap Q_1^+$. Moreover, if $\rho\ge r/4$, then
\begin{align*}
  \int_{Q_{\rho}(z_0)}|\wDu-\mean{\wDu}_{z_0, \rho}|^2 &\le \int_{Q_{\rho}(z_0)}|\wDu-\mean{\wDu}_{z_0, r}|^2\\&\le 4^{n+4}\left(\frac \rho r\right)^{n+4}\int_{Q_r(z_0)}|\wDu-\mean{\wDu}_{z_0, r}|^2,
\end{align*}
thus we may assume $\rho< r/4$ as well.  Let $d:=d_{\rm par}(z_0, Q'_1)=\dist(x_0,B'_1)>0$ and
choose $x_1\in \pa B_d(x_0)\cap B'_1$. The assumption
that $K$ is a cylinder centered on $Q'_1$ implies $z_1:=(x_1,t_0)\in K\cap Q_1'$.

\medskip\noindent \emph{Case 1.} If $\rho\ge d$, then from
$Q_{\rho}(z_0)\subset Q_{2\rho}(z_1)\subset Q_{r/2}(z_1)\subset
Q_r(z_0)$ and Proposition~\ref{prop:par-alm-min-grad-Camp-est-thin}, we deduce
\begin{align*}
  \int_{Q_{\rho}(z_0)}|\wDu-\mean{\wDu}_{z_0, \rho}|^2
  &\le \int_{Q_{2\rho}(z_1)}|\wDu-\mean{\wDu}_{z_1, 2\rho}|^2\\
  &\begin{multlined}\le C(n)\left(\frac \rho r\right)^{n+4}\int_{Q_{r/2}(z_1)}|\wDu-\mean{\wDu}_{z_1, r/2}|^2 \\+C(n, \al, K)\|u\|^2_{W^{1, 0}_2(Q_1)}r^{n+2+2\be}
  \end{multlined}\\
  &\begin{multlined}\le C(n)\left(\frac \rho
     r\right)^{n+4}\int_{Q_r(z_0)}|\wDu-\mean{\wDu}_{z_0,
     r}|^2\\+C(n, \al, K)\|u\|^2_{W^{1, 0}_2(Q_1)}r^{n+2+2\be},
 \end{multlined}
\end{align*}
which gives \eqref{eq:par-alm-min-grad-Camp-est} in this case.

\medskip\noindent \emph{Case 2.} Now we suppose $d>\rho$. If
$d>r$, then $Q_{r}(z_0)\subset Q_1^+$ and since $u$ is almost caloric
in $Q_1^+$, \eqref{eq:par-alm-min-grad-Camp-est} follows from \eqref{eq:alm-cal-Camp-est-improved} in the proof of Theorem~\ref{thm:alm-cal-reg}.
Thus we may assume $d\leq r$. Then $Q_d(z_0)\subset Q_1^{+}$, and
hence again by \eqref{eq:alm-cal-Camp-est-improved}, we have
\begin{multline*}
  \int_{Q_{\rho}(z_0)}|\wDu-\mean{\wDu}_{z_0, \rho}|^2 \le C(n,\al)\left(\frac \rho
    d\right)^{n+4}\int_{Q_d(z_0)}|\wDu-\mean{\wDu}_{z_0,d}|^2\\+C(n, \al, K)\|u\|^2_{W^{1,0}_2(Q_1)}d^{n+2+2\be}.
\end{multline*}
We need to consider further subcases.

\medskip\noindent\emph{Case 2.1.} If $r/4\le d$, then (since also
$d\leq r$)
$$\int_{Q_d(z_0)}|\wDu-\mean{\wDu}_{z_0, d}|^2 \le 4^{n+4}\left(\frac{d}{r}\right)^{n+4}\int_{Q_r(z_0)}|\wDu-\mean{\wDu}_{z_0, r}|^2$$
and combined with the previous inequality, we obtain
\eqref{eq:par-alm-min-grad-Camp-est} in this subcase.

\medskip\noindent \emph{Case 2.2.} If $d<r/4$, then we also have
\begin{align*}
  \int_{Q_d(z_0)}|\wDu-\mean{\wDu}_{z_0, d}|^2
  &\le \int_{Q_{2d}(z_1)}|\wDu-\mean{\wDu}_{z_1, 2d}|^2 \\
  &\begin{multlined}
    \le C(n)\left(\frac{d}r\right)^{n+4}\int_{Q_{r/2}(z_1)}|\wDu-\mean{\wDu}_{z_1, r/2}|^2\\
    +C(n, \al, K)\|u\|^2_{W^{1,0}_2(Q_1)}r^{n+2+2\be}
  \end{multlined}\\
  &\begin{multlined}\le C(n)\left(\frac{d}{r}\right)^{n+4}\int_{Q_{r}(z_0)}|\wDu-\mean{\wDu}_{z_0,
     r}|^2\\+C(n, \al, K)\|u\|^2_{W^{1, 0}_2(Q_1)}r^{n+2+2\be}.
 \end{multlined}
\end{align*}
Hence, the estimate \eqref{eq:par-alm-min-grad-Camp-est} has been
established in all possible cases.

Now we apply Lemma~\ref{lem:HL} to \eqref{eq:par-alm-min-grad-Camp-est} to obtain for $0<\rho<r<r_0$,
\begin{multline*}
  \int_{Q_{\rho}(z_0)}|\wDu-\mean{\wDu}_{z_0, \rho}|^2 \le C(n,\al)\biggl[
  \Bigl(\frac{\rho}{r}\Bigr)^{n+2+2\be}\int_{Q_{r}(z_0)}|\wDu-\mean{\wDu}_{z_0,
    r}|^2\\+C(n, \al, K)\|u\|_{W^{1, 0}_2(Q_1)}^2\rho^{n+2+2\be}\biggr].
\end{multline*}
Taking $r\nearrow r_0=r_0(n, \al, K)$ gives
\begin{align*}
  \int_{Q_{\rho}(z_0)}|\wDu-\mean{\wDu}_{z_0, \rho}|^2 &\le C(n, \al, K)\|u\|_{W^{1, 0}_2(Q_1)}^2\rho^{n+2+2\be}.
\end{align*}
Finally, by the Campanato space embedding (Lemma~\ref{lem:par-camp-space-embed}), we conclude that
$$\wDu\in H^{\be,\be/2}(K)$$ with
\begin{equation*}
  \|\wDu\|_{H^{\be,\be/2}(K)}\le C(n, \al, K)\|u\|_{W^{1, 0}_2(Q_1)}.\qedhere
\end{equation*}
\end{proof}

\begin{proof}[Proof of Theorem~\ref{mthm:I}] Parts (1) and (2) are contained in Theorems~\ref{thm:par-alm-min-holder} and \ref{thm:par-alm-min-grad-holder}, respectively.
\end{proof}

%%%%%%%%%%%%%%%%%%%%%%%%%%%%%%%%%%%%%%%%%

\section{Almost minimizers with variable coefficients}\label{sec:alm-min-var}

In this section we investigate almost minimizers for the parabolic $A$-Signorini problem. We extend the results for almost minimizers for the parabolic Signorini problem (in the case of $A(z)\equiv I$) to the framework of variable coefficients by following the idea in \cite{JeoPetSVG20}.

Note that for $z,z_0\in Q_1$, $\xi\in \R^n$ and $\xi\neq0$ $$
\left(1-C\|z-z_0\|^\al\right)\le \frac{\mean{A(z_0)\xi,\xi}}{\mean{A(z)\xi,\xi}}\le \left(1+C\|z-z_0\|^\al\right)
$$
for some constant $C>0$ depending on $\la$, $\Ld$ and $\|A\|_{H^{\al,\al/2}(Q_1)}$. Then we can rewrite \eqref{eq:par-var-alm-min-def} in the form \emph{with frozen coefficients}
\begin{equation}\label{eq:par-var-alm-min-def-frozen}
    \int_{F_r(z_0)}(1-C_1r^\al)\mean{A(z_0)\D U,\D U}+2\pa_tU(U-V)
    \le (1+C_1r^\al)\int_{F_r(z_0)}\mean{A(z_0)\D V,\D V}
\end{equation}
for a constant $C_1>0$ depending on $\la$, $\Lambda$ and $\|A\|_{H^{\al,\al/2}(Q_1)}$.

For the simplicity of the tracking of the constants, we take $M>0$ such that $$
\|A\|_{H^{\al,\al/2}(Q_1)},\la^{-1},\Ld,C_1\le M.
$$

\subsection{Coordinate transformations}
To use the results in previous sections for almost minimizers in the case of $A\equiv I$, we make use of coordinate transformations, introduced in \cite{JeoPetSVG20}, to straighten $A(z_0)$.

For this purpose, we first denote (for the notational simplicity) 
$$
\mfa_{z_0}=A^{1/2}(z_0),\quad z_0\in Q_1.
$$
Note that 
$$
\mean{\mfa_{z_0}\xi,\xi}=|\mfa_{z_0}\xi|^2,\quad \xi\in\R^n,
$$
and $\mfa_{z_0}$ is a symmetric positive definite $n\times n$ matrix with eigenvalues between $\la^{1/2}$ and $\Ld^{1/2}$, and the mapping $z_0\mapsto \mfa_{z_0}$ is $H^{\al,\al/2}$-continuous for $z_0\in Q_1$. 

We next define, for $z_0=(x_0,t_0)\in Q_1$, an affine transformation $T_{z_0}:\R^n\to\R^n$ by $$
T_{z_0}(x)=\mfa_{z_0}^{-1}(x-x_0)
$$
so that $E_r(z_0)=T_{z_0}^{-1}(B_r)$. We further define hyperplanes in $\R^n$ 
$$
\Pi:=\R^{n-1}\times\{0\}\subset\R^n \quad\text{and}\quad \Pi_{z_0}:=T_{z_0}(\Pi).
$$
In general, $\Pi_{z_0}$ is tilted with respect to $\Pi$, which is problematic when we try to use the results available for almost minimizers with $A\equiv I$. It can be rectified with the help of a family of orthogonal transformations $O_{z_0}$, $z_0\in Q_1$, as in Section 2 in \cite{JeoPetSVG20}. That is, we consider a rotation $O_{z_0}:\R^n\to\R^n$ such that 
$z_0\mapsto O_{z_0}$ is $H^{\al,\al/2}$-continuous and $O_{z_0}^{-1}(\Pi_{z_0})=\Pi$.

\noindent If we let
$$
\bar{\mfa}_{z_0}=\mfa_{z_0}O_{z_0}\quad\text{and}\quad \bar{T}_{z_0}=O_{z_0}^{-1}\circ T_{z_0},
$$
then we still have (since $O_{z_0}$ is a rotation) $$
\bar{T}_{z_0}(x)=\bar{\mfa}_{z_0}^{-1}(x-x_0)\quad\text{and}\quad E_r(z_0)=\bar{T}_{z_0}^{-1}(B_r).
$$
Now, we define the “deskewed”
version of $U$ at $z_0$ $$
u_{z_0}(x,t):=U(\bar{T}_{z_0}^{-1}(x),t+t_0)=U(\bar{\mfa}_{z_0}x+x_0,t+t_0).
$$
We will prove in Lemma~\ref{lem:par-var-alm-min-transf-property} that the transformed function $u_{z_0}$ satisfies the almost Signorini property at the origin, allowing us to use the results in the previous sections. Before stating and proving the lemma, we write the following (basic) change of variable formulas that will be used multiple times in this section: \begin{align}
\label{eq:cov1}
\int_{F_r(z_0)\times F_r(z_0)}|U(z)-U(w)|^2\,dzdw&=(\det\mfa_{z_0})^2\int_{Q_r\times Q_r}|u_{z_0}(z)-u_{z_0}(w)|^2\,dzdw,\\
\label{eq:cov2}
\int_{F_r(z_0)}\pa_tU(U-V)&=\det\mfa_{z_0}\int_{Q_r}\pa_tu_{z_0}(u_{z_0}-v_{z_0}),\\
\label{eq:cov3}\int_{F_r(z_0)}\mean{A(z_0)\D U,\D U}&=\det\mfa_{z_0}\int_{Q_r}|\D u_{z_0}|^2.
\end{align}

\begin{lemma}\label{lem:par-var-alm-min-transf-property}
Let $U$ be an almost minimizer of the parabolic $A$-Signorini problem in $Q_1$. Then for any $z_0\in Q'_1$ with $F_R(z_0)\subset Q_1$, $u_{z_0}$ satisfies the almost parabolic Signorini property at $0$ in $Q_R$ (with gauge function $\omega(r)=Mr^\al$).
\end{lemma}

\begin{proof}
For $0<r<R$, let $v_{z_0}\in W^{1,0}_2(Q_r)$ be such that $v_{z_0}\ge 0$ on $Q'_r$ and $v_{z_0}-u_{z_0}\in L^2(-r^2,0;W^{1,2}_0(B_r))$. Then $V(x,t):=v_{z_0}(\bar{T}_{z_0}(x),t-t_0)$ satisfies $V\ge 0$ on $F'_r(z_0)$ and $V-U\in L^2(t_0-r^2,t_0;W^{1,2}_0(E_r(z_0)))$, i.e., $V$ is a valid competitor for the almost minimizer $U$. Thus \begin{align*}
    \hspace{2em}&\hspace{-2em}\int_{Q_r}(1-Mr^\al)|\D u_{z_0}|^2+2\pa_tu_{z_0}(u_{z_0}-v_{z_0})\\
    &=\det\mfa_{z_0}^{-1}\int_{F_r(z_0)}(1-Mr^\al)\mean{A(z_0)\D U,\D U}+2\pa_tU(U-V)\\
    &\le (1+Mr^\al)\det\mfa_{z_0}^{-1}\int_{F_r(z_0)}\mean{A(z_0)\D V,\D V}=(1+Mr^\al)\int_{Q_r}|\D v_{z_0}|^2.\qedhere
\end{align*}
\end{proof}

\subsection{Almost $A$-caloric functions}

\begin{definition}
We say that $U\in W^{1,1}_2(Q_1)$ is an \emph{almost $A$-caloric function} in $Q_1$ if \begin{align*}
\int_{F_r(z_0)}(1-r^\al)\mean{A\D U,\D U}+2\pa_tU(U-V)\le (1+r^\al)\int_{F_r(z_0)}\mean{A\D V,\D V},
\end{align*}
whenever $F_r(z_0)\Subset Q_1$ and $V\in U+L^2(t_0-r^2,t_0;W^{1,2}_0(E_r(z_0)))$.
\end{definition}

Note that as in the parabolic $A$-Signorini problem, the above almost minimizing property can be written in the form with frozen coefficients $$
\int_{F_r(z_0)}(1-Mr^\al)\mean{A(z_0)\D U,\D U}+2\pa_tU(U-V)\le (1+Mr^\al)\int_{F_r(z_0)}\mean{A(z_0)\D V,\D V}.
$$

\begin{lemma}
If $U$ is an almost $A$-caloric function in $Q_1$ and $z_0\in Q_1$ with $F_R(z_0)\subset Q_1$, then $u_{z_0}$ satisfies almost caloric property at $0$ in $Q_R$.
\end{lemma}

\begin{proof}
The proof is similar to that of Lemma~\ref{lem:par-var-alm-min-transf-property}.
\end{proof}

Recall $$
\vp_{z_0}(r,U)=\int_{Q_r(z_0)}r^{n+4}|\D U|^2+\int_{Q_r(z_0)\times Q_r(z_0)}|U(z)-U(w)|^2\,dzdw.
$$
Combininig the preceding lemma and Proposition~\ref{prop:alm-cal-Mor-Camp-est}, we can obtain the following estimates for almost $A$-caloric functions.

\begin{proposition}
\label{prop:var-alm-cal-Mor-Camp-est}
Let $U$ be an almost $A$-caloric function in $Q_1$. Then for any $Q_R(z_0)\Subset Q_1$ and $0<\rho<R$,
\begin{align}
    \label{eq:var-alm-cal-Mor-est} &\vp_{z_0}(\rho,U)\le C\left[\left(\frac\rho R\right)^{2n+6}+R^\al\right]\vp_{z_0}(R,U),\\
    \label{eq:var-alm-cal-Camp-est}&\begin{multlined}[t]\int_{Q_\rho(z_0)}|\D U-\mean{\D U}_{z_0,\rho}|^2\le C\left(\frac\rho R\right)^{n+4}\int_{Q_R(z_0)}|\D U-\mean{\D U}_{z_0,R}|^2\\+CR^\al\int_{Q_R(z_0)}|\D U|^2,
\end{multlined}\end{align}
with $C=C(n,\al,M)$.
\end{proposition}

\begin{proof}
Since $u_{z_0}$ satisfies almost caloric property at $0$ in $Q_r$ with $r:=\Ld^{-1/2}R$, applying Proposition~\ref{prop:alm-cal-Mor-Camp-est} gives that for $0<\rho<r$ \begin{align*}
    &\vp_0(\rho,u_{z_0})\le C(n,\al,M)\left[\left(\frac\rho r\right)^{2n+6}+r^\al\right]\vp_0(r,u_{z_0}),\\
    &\begin{multlined}
      \int_{Q_\rho}|\D u_{z_0}-\mean{\D u_{z_0}}_\rho|^2\le C(n,\al,M)\left(\frac\rho r\right)^{n+4}\int_{Q_r}|\D u_{z_0}-\mean{\D u_{z_0}}_r|^2\\ +C(n,M)r^\al\int_{Q_r}|\D u_{z_0}|^2.
    \end{multlined}
\end{align*}
To rewrite these estimates in terms of $U$, we introduce the modification of the quantity $\vp_{z_0}(r,U)$: $$
\vp_{z_0}^A(r,U):=\int_{F_r(z_0)}r^{n+4}\mean{A(z_0)\D U,\D U}+\int_{F_r(z_0)\times F_r(z_0)}|U(z)-U(w)|^2\,dzdw.
$$
By applying the change of variable formulas \eqref{eq:cov1}--\eqref{eq:cov3} and using that $\la^{n/2}\le \det \mfa_{z_0}\le \Ld^{n/2}$, we deduce \begin{align}
    \label{eq:var-alm-cal-Mor-est-A}&\vp_{z_0}^A(\rho,U)\le C(n,\al,M)\left[\left(\frac\rho r\right)^{2n+6}+r^\al\right]\vp_{z_0}^A(r,U),\\
    &\label{eq:var-alm-cal-Camp-est-A}\begin{multlined}[t]
      \int_{F_\rho(z_0)}|\bar{\mfa}_{z_0}\D U-\mean{\bar{\mfa}_{z_0}\D U}_{F_\rho(z_0)}|^2\\
      \qquad\qquad \le C(n,\al,M)\left(\frac\rho r\right)^{n+4}\int_{F_r(z_0)}|\bar{\mfa}_{z_0}\D U-\mean{\bar{\mfa}_{z_0}\D U}_{F_r(z_0)}|^2\\
      +C(n,M)r^\al\int_{F_r(z_0)}\mean{A(z_0)\D U,\D U},
    \end{multlined}
\end{align}
for $0<\rho<r$. To show that \eqref{eq:var-alm-cal-Mor-est-A}--\eqref{eq:var-alm-cal-Camp-est-A} imply \eqref{eq:var-alm-cal-Mor-est}--\eqref{eq:var-alm-cal-Camp-est}, we first consider the case $$
0<\rho<\left(\la/\Ld\right)^{1/2}R.
$$
Applying \eqref{eq:var-alm-cal-Mor-est-A}--\eqref{eq:var-alm-cal-Camp-est-A} with $\la^{-1/2}\rho$ and $\Ld^{-1/2}R$ in place of $\rho$ and $r$, and using the inclusions $$
Q_\rho(z_0)\subset F_{\la^{-1/2}\rho}(z_0)\subset F_{\Ld^{-1/2}R}(z_0)\subset Q_R(z_0)
$$
and the ellipticity of $A(z_0)$, we obtain \eqref{eq:var-alm-cal-Mor-est}--\eqref{eq:var-alm-cal-Camp-est} in this case.

On the other hand, when $\left(\la/\Ld\right)^{1/2}R\le \rho\le R$, we simply have \begin{align*}
    \vp_{z_0}(\rho,U)&\le \left(\frac\Ld\la\right)^{n+3}\left(\frac\rho R\right)^{2n+6}\vp_{z_0}(R,U),\\
    \int_{Q_\rho(z_0)}|\D U-\mean{\D U}_{z_0,\rho}|^2&\le \left(\frac\Ld\la\right)^{\frac{n+4}2}\left(\frac\rho R\right)^{n+4}\int_{Q_R(z_0)}|\D U-\mean{\D U}_{z_0,R}|^2.\qedhere
\end{align*}
\end{proof}

\begin{theorem}\label{thm:alm-A-cal-reg}
Let $U$ be an almost $A$-caloric function in $Q_1$. Then 
\begin{enumerate}
    \item $U\in H^{\si,\si/2}(Q_1)$ for all $\si\in (0,1)$;
    \item $\D U\in H^{\al/2,\al/4}(Q_1)$.
\end{enumerate}
\end{theorem}

\begin{proof}
With Proposition~\ref{prop:var-alm-cal-Mor-Camp-est} at hand, the proof of Theorem~\ref{thm:alm-A-cal-reg}
is similar to that of Theorem~\ref{thm:alm-cal-reg}. 
\end{proof}

%%%%%%%%%%%%%%%%%%%%%%%%%%%%%%%%%%%%%%%%%%%%%

\subsection{Regularity of almost minimizers with variable coefficients}

In this subsection, we study almost minimizers for the parabolic $A$-Signorini problem, and establish Theorem~\ref{mthm:II}.

\begin{proposition}
Let $U$ be an almost minimizer for the parabolic $A$-Signorini problem in $Q_1$, and $Q_R(z_0)\Subset Q_1$ with $z_0\in Q'_1$. Then, there is $C=C(n,\al,M)>1$ such that \begin{align}
    \label{eq:var-par-alm-min-Mor-est-thin}
    \vp_{z_0}(\rho, U)\le C\left[\left(\frac\rho R\right)^{2n+6}+R^\al\right]\vp_{z_0}(R,U),\quad 0<\rho<R.
\end{align}
\end{proposition}

\begin{proof}
Since $u_{z_0}$ satisfies the almost parabolic Signorini property at $0$ in $Q_r$ for $r=\Ld^{-1/2}R$, we have by Proposition~\ref{prop:par-alm-min-Mor-est-thin} $$
\vp_0(\rho, u_{z_0})\le C(n,\al,M)\left[\left(\frac\rho r\right)^{2n+6}+r^\al\right]\vp_0(r,u_{z_0}),\quad 0<\rho<r.
$$
By transforming back from $u_{z_0}$ to $U$ as we did in Proposition~\ref{prop:var-alm-cal-Mor-Camp-est}, we obtain \eqref{eq:var-par-alm-min-Mor-est-thin}.
\end{proof}

\begin{theorem}\label{thm:var-par-alm-min-holder}
Let $U$ be an almost minimizer of the parabolic $A$-Signorini problem in $Q_1$. Then $u\in H^{\si,\si/2}(Q_1)$ for every $\si\in (0,1)$. Moreover, for any $K\Subset Q_1$, $$\|U\|_{H^{\si,\si/2}(K)}\le C(K,n,\al,\si,M)\|U\|_{W^{1,0}_2(Q_1)}.
$$
\end{theorem}

\begin{proof}
With \eqref{eq:var-alm-cal-Mor-est} and \eqref{eq:var-par-alm-min-Mor-est-thin} at hand, the proof is similar to that of Theorem~\ref{thm:par-alm-min-holder}.
\end{proof}

\begin{proposition}
Let $U$ be an almost minimizer of the parabolic $A$-Signorini problem in $Q_1$. Define $$
\wDU(x',x_n,t)=\begin{cases}
  \D U(x',x_n,t),& x_n\ge 0\\
  \D U(x',-x_n,t),& x_n<0.
\end{cases}
$$
Then for any $K\Subset Q_1$, there is $R_0=R_0(n,\al,K,M)>0$ such that if $z_0\in K'=K\cap Q'_1$, then 
\begin{multline}
\label{eq:var-par-alm-min-grad-Camp-est}
\int_{Q_\rho(z_0)}|\wDU-\mean{\wDU}_{z_0,\rho}|^2\le C(n,M)\left(\frac\rho R\right)^{n+4}\int_{Q_R(z_0)}|\wDU-\mean{\wDU}_{z_0,R}|^2\\ +C(n,\al,K,M)\|U\|_{W^{1,0}_2(Q_1)}^2R^{n+2+2\be},
\end{multline}
for $0<\rho<R<R_0$ with $\be=\frac\al{4(2n+4+\al)}$.
\end{proposition}

\begin{proof}
Since $u_{z_0}$ satisfies the almost parabolic Signorini property at $0$ in $Q_{R_1}$ for some $R_1=R_1(K,M)>0$, applying Proposition~\ref{prop:par-alm-min-grad-Camp-est-thin} gives \begin{equation*}\begin{aligned}
    \int_{Q_\rho}|\widehat{\D u_{z_0}}-\mean{\widehat{\D u_{z_0}}}_\rho|^2&\begin{multlined}[t]\le C(n)\left(\frac\rho R\right)^{n+4}\int_{Q_R}|\widehat{\D u_{z_0}}-\mean{\widehat{\D u_{z_0}}}_R|^2\\
    +C(n,\al,K,M)\|u_{z_0}\|_{W^{1,0}_2(Q_{R_1})}^2R^{n+2+2\be},\end{multlined}\\
    & \begin{multlined}[t]\le C(n)\left(\frac\rho R\right)^{n+4}\int_{Q_R}|\widehat{\D u_{z_0}}-\mean{\widehat{\D u_{z_0}}}_R|^2\\
    +C(n,\al,K,M)\|U\|_{W^{1,0}_2(Q_1)}^2R^{n+2+2\be},
\end{multlined}\end{aligned}\end{equation*}
for $0<\rho<R<R_2=R_2(n,\al,K,M)$. By the even symmetry of $\widehat{\D u_{z_0}}$, \begin{multline*}
\int_{Q_\rho^+}|\D u_{z_0}-\mean{\D u_{z_0}}_{Q_\rho^+}|^2\le C(n)\left(\frac\rho R\right)^{n+4}\int_{Q_R^+}|\D u_{z_0}-\mean{\D u_{z_0}}_{Q_R^+}|^2\\ +C(n,\al,K,M)\|U\|^2_{W^{1,0}_2(Q_1)}R^{n+2+2\be}.
\end{multline*}
Using that $\bar{T}_{z_0}^{-1}(B_r^+)\times(t_0-r^2,t_0]=F_r^+(z_0)$ for any $r>0$, the above inequality becomes \begin{multline*}
\int_{F_\rho^+(z_0)}|\bar{\mfa}_{z_0}\D U-\mean{\bar{\mfa}_{z_0}\D U}_{F_\rho^+(z_0)}|^2\\
\qquad\qquad\le C(n)\left(\frac\rho R\right)^{n+4}\int_{F_R^+(z_0)}|\bar{\mfa}_{z_0}\D U-\mean{\bar{\mfa}_{z_0}\D U}_{F_R^+(z_0)}|^2\\ +C(n,\al,K,M)\|U\|^2_{W^{1,0}_2(Q_1)}R^{n+2+2\be}.
\end{multline*}
Repeating the argument that \eqref{eq:var-alm-cal-Camp-est-A} implies \eqref{eq:var-alm-cal-Camp-est} in the proof of Proposition~\ref{prop:var-alm-cal-Mor-Camp-est}, we obtain \begin{multline*}
\int_{Q_\rho^+(z_0)}|\D U-\mean{\D U}_{Q_\rho^+(z_0)}|^2\\
\qquad\qquad\le C(n,M)\left(\frac\rho R\right)^{n+4}\int_{Q_R^+(z_0)}|\D U-\mean{\D U}_{Q^+_R(z_0)}|^2\\ +C(n,\al,K,M)\|U\|^2_{W^{1,0}_2(Q_1)}R^{n+2+2\be}.
\end{multline*}
By the even symmetry of $\wDU$, this inequality implies \eqref{eq:var-par-alm-min-grad-Camp-est}.
\end{proof}

\begin{theorem}\label{thm:var-par-alm-min-grad-holder}
Let $U$ be an almost minimizer for the parabolic $A$-Signorini problem in $Q_1$. Then $$\wDU\in H^{\be,\be/2}(Q_1)\quad\text{with }\be=\frac\al{4(2n+4+\al)}.$$
Morevoer, for any $K\Subset Q_1$, we have $$
\|\wDU\|_{H^{\be,\be/2}(K)}\le C(n,\al,K,M)\|U\|^2_{W^{1,0}_2(Q_1)}.
$$
\end{theorem}

\begin{proof}
Since $U$ is almost $A$-caloric in $Q_1^\pm$, we have by \eqref{eq:var-alm-cal-Camp-est} and Theorem~\ref{thm:alm-A-cal-reg}
\begin{multline*}
    \int_{Q_\rho(z_0)}|\D U-\mean{\D U}_{z_0,\rho}|^2\le C(n,\al,M)\left(\frac\rho r\right)^{n+4}\int_{Q_r(z_0)}|\D U-\mean{\D U}_{z_0,r}|^2\\ +C(n,K,\al,M)\|U\|_{W^{1,0}_2(Q_1)}^2r^{n+2+\al},
\end{multline*}
whenever $z_0\in K$ and $0<\rho<r<r_0=r_0(n,K,\al,M)$ with $Q_r(z_0)\Subset Q_1^\pm$. As we have seen in the proof of Theorem~\ref{thm:par-alm-min-grad-holder}, the above estimate, along with \eqref{eq:var-par-alm-min-grad-Camp-est}, gives Theorem~\ref{thm:var-par-alm-min-grad-holder}.
\end{proof}

We finish this section by formally proving Theorem~\ref{mthm:II}.

\begin{proof}[Proof of Theorem~\ref{mthm:II}] Theorem~\ref{mthm:II} follows by combining Theorems~\ref{thm:var-par-alm-min-holder} and \ref{thm:var-par-alm-min-grad-holder}.
\end{proof}

\appendix

%%%%%%%%%%%%%%%%%%%%%%%%%%%%%%%%%%%%%%%%%%%%%

\section{An example of almost minimizers}\label{sec:appen-ex}

\begin{example}\label{ex:var-par-alm-min-drift}
Let $U$ be a solution of the parabolic $A$-Signorini problem in $Q_1$ with the velocity field $b\in L^\infty(-1,0\,;\,L^p(B_1))$, $p>n$: \begin{align*}
    -\dv(A\D U)+\mean{b(x,t),\D U}+\pa_tU=0&\quad\text{in }Q_1^\pm\\
    \mean{A\D U,\nu^+}+\mean{A\D U,\nu^-}\ge 0,\ U\ge 0,\ U(\mean{A\D U,\nu^+}+\mean{A\D U,\nu^-})=0&\quad\text{on }Q'_1,
\end{align*}
where $\nu^\pm=\mp e_n$ and $\mean{A\D U,\nu^\pm}$ on $Q'_1$ are understood as the limits from inside $Q_1^\pm$. We interpret this in the weak sense that $U$ satisfies for a.e. $t\in(-1,0)$ the variational inequality $$
\int_{B_1}\mean{A\D U,\D(V-U)}+\mean{b,\D U}(V-U)+\pa_tU(V-U)\ge 0,
$$
for any competitor $V\in W^{1,2}(B_1)$ such that $V\ge 0$ on $B'_1$ and $V=U$ on $\pa B_1$. Then $U$ is an almost minimizer for the parabolic $A$-Signorini problem in $Q_1$ with a gauge function $\om(r)=Cr^{1-n/p}$, for some $C>0$ depending on $n$, $p$, $\sup_{-1<t<0}\left(\|b(\cdot,t)\|_{L^p(B_1)}\right)$.

\end{example}

\begin{proof}
For $F_r(z_0)=F_r(x_0,t_0)\Subset Q_1$, let $V\in W^{1,0}_2(F_r(z_0))$ be such that $V\ge 0$ on $F'_r(z_0)$ and $V-U\in L^2(t_0-r^2,t_0;W^{1,2}_0(E_r(z_0))$. By extending $V$ to $B_1\times(t_0-r^2,t_0]$ as equal to $U$ in $(B_1\sm E_r(z_0))\times(t_0-r^2,t_0]$ and applying the variational inequality, we get $$
\int_{F_r(z_0)}\mean{A\D U,\D(V-U)}+\mean{b,\D U}(V-U)+\pa_tU(V-U)\ge 0.
$$
It follows that \begin{multline}
    \label{eq:ex-var-par-alm-min-drift-1}
    \int_{F_r(z_0)}\mean{A\D U,\D U}+2\pa_tU(U-V)\\
    \le \int_{F_r(z_0)}\mean{A\D U,\D U}+2\mean{A\D U,\D(V-U)}+2\mean{b,\D U}(V-U).
\end{multline}
To estimate the terms in the right-hand side of \eqref{eq:ex-var-par-alm-min-drift-1}, we observe that 
\begin{multline}
    \label{eq:ex-var-par-alm-min-drift-2}
    \int_{F_r(z_0)}\mean{A\D U,\D U}+2\mean{A\D U,\D(V-U)}\\
    =\int_{F_r(z_0)}2\mean{\D U,\D V}-\mean{A\D U,\D U}\le \int_{F_r(z_0)}\mean{A\D V,\D V}.
\end{multline}
Moreover, for $L:=\sup\{\|b(\cdot,t)\|_{L^p(B_1)}\,:\,-1<t<0\}$, $$
\int_{F_r(z_0)}2\mean{b,\D U}(V-U)\le \int_{t_0-r^2}^{t_0}2L\|\D U\|_{L^2(E_r(z_0))}\|V-U\|_{L^{p^*}(E_r(z_0))}\,dt,
$$
with $p^*=2p/(p-2)$. Since $(V-U)(\cdot,t)\in W^{1,2}_0(E_r(z_0))$, Sobolev's inequality yields $$
\|V-U\|_{L^{p^*}(E_r(z_0))}\le C_{n,p}r^\g\|\D(V-U)\|_{L^2(E_r(z_0))},
$$
with $\g=1-n/p$. Thus, 
\begin{align}\label{eq:ex-var-par-alm-min-drift-3}
\hspace{2em}&\hspace{-2em}\int_{Q_r(z_0)}2\mean{b,\D U}(V-U)\\
    &\nonumber\le \int_{t_0-r^2}^{t_0}C_{n,p,L}r^\g\|\D U\|_{L^2(E_r(z_0))}\|\D(V-U)\|_{L^2(E_r(z_0))}\,dt\\
    &\nonumber\le Cr^\g\int_{t_0-r^2}^{t_0}\left(\|\D U\|^2_{L^2(E_r(z_0))}+\|\D(V-U)\|^2_{L^2(E_r(z_0))}\right)\,dt\\
    &\nonumber\le Cr^\g\int_{F_r(z_0)}\left(|\D U|^2+|\D V|^2\right)\\
    &\nonumber\le C_{n,p,L,M}r^\g\int_{F_r(z_0)}\mean{A\D U,\D U}+\mean{A\D V,\D V}
\end{align}
Now, by \eqref{eq:ex-var-par-alm-min-drift-1}, \eqref{eq:ex-var-par-alm-min-drift-2} and \eqref{eq:ex-var-par-alm-min-drift-3}, we conclude that \begin{equation*}
\int_{F_r(z_0)}(1-Cr^\g)\mean{A\D U,\D U}+2\pa_tU(U-V)\le (1+Cr^\g)\int_{F_r(z_0)}\mean{A\D V,\D V}.\qedhere
\end{equation*}
\end{proof}

%%%%%%%%%%%%%%%%%%%%%%%%%%%%%%%%%%%%%%%%%%


\begin{bibdiv}
\begin{biblist}

\bib{Anz83}{article}{
   author={Anzellotti, Gabriele},
   title={On the $C^{1,\alpha }$-regularity of $\omega $-minima of
   quadratic functionals},
   language={English, with Italian summary},
   journal={Boll. Un. Mat. Ital. C (6)},
   volume={2},
   date={1983},
   number={1},
   pages={195--212},
   review={\MR{718371}},
}

\bib{ArkUra96}{article}{
   author={Arkhipova, A.},
   author={Uraltseva, N.},
   title={Sharp estimates for solutions of a parabolic Signorini problem},
   journal={Math. Nachr.},
   volume={177},
   date={1996},
   pages={11--29},
   issn={0025-584X},
   review={\MR{1374941}},
   doi={10.1002/mana.19961770103},
}

\bib{AthCaf04}{article}{
   author={Athanasopoulos, I.},
   author={Caffarelli, L. A.},
   title={Optimal regularity of lower dimensional obstacle problems},
   language={English, with English and Russian summaries},
   journal={Zap. Nauchn. Sem. S.-Peterburg. Otdel. Mat. Inst. Steklov.
   (POMI)},
   volume={310},
   date={2004},
   number={Kraev. Zadachi Mat. Fiz. i Smezh. Vopr. Teor. Funkts. 35
   [34]},
   pages={49--66, 226},
   issn={0373-2703},
   translation={
      journal={J. Math. Sci. (N.Y.)},
      volume={132},
      date={2006},
      number={3},
      pages={274--284},
      issn={1072-3374},
   },
   review={\MR{2120184}},
   doi={10.1007/s10958-005-0496-1},
}

\bib{AthCafMil18}{article}{
   author={Athanasopoulos, Ioannis},
   author={Caffarelli, Luis},
   author={Milakis, Emmanouil},
   title={On the regularity of the non-dynamic parabolic fractional obstacle
   problem},
   journal={J. Differential Equations},
   volume={265},
   date={2018},
   number={6},
   pages={2614--2647},
   issn={0022-0396},
   review={\MR{3804726}},
   doi={10.1016/j.jde.2018.04.043},
}

\bib{AthCafSal08}{article}{
   author={Athanasopoulos, I.},
   author={Caffarelli, L. A.},
   author={Salsa, S.},
   title={The structure of the free boundary for lower dimensional obstacle
   problems},
   journal={Amer. J. Math.},
   volume={130},
   date={2008},
   number={2},
   pages={485--498},
   issn={0002-9327},
   review={\MR{2405165}},
   doi={10.1353/ajm.2008.0016},
}



\bib{BanDanGarPet20}{article}{
   author={Banerjee, A.},
   author={Danielli, D.},
   author={Garofalo, N.},
   author={Petrosyan, A.},
   title={The regular free boundary in the thin obstacle problem for
   degenerate parabolic equations},
   journal={Algebra i Analiz},
   volume={32},
   date={2020},
   number={3},
   pages={84--126},
   issn={0234-0852},
   translation={
      journal={St. Petersburg Math. J.},
      volume={32},
      date={2021},
      number={3},
      pages={449--480},
      issn={1061-0022},
   },
   review={\MR{4099095}},
   doi={10.1090/spmj/1656},
}
	


\bib{BanDanGarPet21}{article}{
   author={Banerjee, Agnid},
   author={Danielli, Donatella},
   author={Garofalo, Nicola},
   author={Petrosyan, Arshak},
   title={The structure of the singular set in the thin obstacle problem for
   degenerate parabolic equations},
   journal={Calc. Var. Partial Differential Equations},
   volume={60},
   date={2021},
   number={3},
   pages={Paper No. 91, 52},
   issn={0944-2669},
   review={\MR{4249869}},
   doi={10.1007/s00526-021-01938-2},
}
		

	










\bib{BanSVGZel17}{article}{
   author={Banerjee, Agnid},
   author={{Smit Vega Garcia}, Mariana},
   author={Zeller, Andrew K.},
   title={Higher regularity of the free boundary in the parabolic Signorini
   problem},
   journal={Calc. Var. Partial Differential Equations},
   volume={56},
   date={2017},
   number={1},
   pages={Art. 7, 26},
   issn={0944-2669},
   review={\MR{3592762}},
   doi={10.1007/s00526-016-1103-7},
} 

\bib{Cam66}{article}{
   author={Campanato, Sergio},
   title={Equazioni paraboliche del secondo ordine e spazi {$\mathcal{L}^{2,\,\theta }\,(\Omega ,\,\delta )$}},
   journal={Ann. Mat. Pura Appl. (4)},
   volume={73},
   date={1966},
   pages={55--102},
   issn={0003-4622},
   review={\MR{213737}},
   doi={10.1007/BF02415082},
 }

\bib{ConTan04}{book}{
   author={Cont, Rama},
   author={Tankov, Peter},
   title={Financial modelling with jump processes},
   series={Chapman \& Hall/CRC Financial Mathematics Series},
   publisher={Chapman \& Hall/CRC, Boca Raton, FL},
   date={2004},
   pages={xvi+535},
   isbn={1-5848-8413-4},
   review={\MR{2042661}},
}

\bib{DanGarPetTo17}{article}{
   author={Danielli, Donatella},
   author={Garofalo, Nicola},
   author={Petrosyan, Arshak},
   author={To, Tung},
   title={Optimal regularity and the free boundary in the parabolic
   Signorini problem},
   journal={Mem. Amer. Math. Soc.},
   volume={249},
   date={2017},
   number={1181},
   pages={v + 103},
   issn={0065-9266},
   isbn={978-1-4704-2547-0},
   isbn={978-1-4704-4129-6},
   review={\MR{3709717}},
   doi={10.1090/memo/1181},
 }

\bib{DuvLio76}{book}{
   author={Duvaut, G.},
   author={Lions, J.-L.},
   title={Inequalities in mechanics and physics},
   note={Translated from the French by C. W. John;
     Grundlehren der Mathematischen Wissenschaften, 219},
   publisher={Springer-Verlag, Berlin-New York},
   date={1976},
   pages={xvi+397},
   isbn={3-540-07327-2},
   review={\MR{0521262}},
}


\bib{GarPet09}{article}{
   author={Garofalo, Nicola},
   author={Petrosyan, Arshak},
   title={Some new monotonicity formulas and the singular set in the lower
   dimensional obstacle problem},
   journal={Invent. Math.},
   volume={177},
   date={2009},
   number={2},
   pages={415--461},
   issn={0020-9910},
   review={\MR{2511747}},
   doi={10.1007/s00222-009-0188-4},
}



\bib{Hab14}{article}{
   author={Habermann, Jens},
   title={Vector-valued parabolic {$\omega$}-minimizers},
   journal={Adv. Differential Equations},
   volume={19},
   date={2014},
   number={11-12},
   pages={1067--1136},
 }


\bib{HanLin97}{book}{
   author={Han, Qing},
   author={Lin, Fanghua},
   title={Elliptic partial differential equations},
   series={Courant Lecture Notes in Mathematics},
   volume={1},
   publisher={New York University, Courant Institute of Mathematical
   Sciences, New York; American Mathematical Society, Providence, RI},
   date={1997},
   pages={x+144},
   isbn={0-9658703-0-8},
   isbn={0-8218-2691-3},
   review={\MR{1669352}},
 }
 

\bib{JeoPet21}{article}{
   author={Jeon, Seongmin},
   author={Petrosyan, Arshak},
   title={Almost minimizers for the thin obstacle problem},
   journal={Calc. Var. Partial Differential Equations},
   volume={60},
   date={2021},
   pages={Paper No. 124, 59},
   issn={0944-2669},
   review={\MR{4277328}},
   doi={10.1007/s00526-021-01986-8},
 }
 
\bib{JeoPetSVG20}{article}{
   author={Jeon, Seongmin},
   author={Petrosyan, Arshak},
   author={Smit Vega Garcia, Mariana},
   title={Almost minimizers for the thin obstacle problem with variable coefficients},
   pages={53},
   date={2020},
   status={preprint},
   eprint={\arXiv{2007.07349}},
 }

\bib{PetZel19}{article}{
   author={Petrosyan, Arshak},
   author={Zeller, Andrew},
   title={Boundedness and continuity of the time derivative in the
     parabolic Signorini problem},
   journal={Math. Res. Let.},
   volume={26},
   number={1},
   pages={281--292},
   date={2019},
   doi={10.4310/MRL.2019.v26.n1.a13},
   %eprint={\arXiv{1512.09173}},
}


\bib{Sig59}{article}{
   author={Signorini, A.},
   title={Questioni di elasticit\`{a} non linearizzata e semilinearizzata},
   language={Italian},
   journal={Rend. Mat. e Appl. (5)},
   volume={18},
   date={1959},
   pages={95--139},
   review={\MR{0118021}},
}

\bib{Ura85}{article}{
   author={Ural\cprime tseva, N. N.},
   title={H\"{o}lder continuity of gradients of solutions of parabolic equations with boundary conditions of {S}ignorini type},
   journal={Dokl. Akad. Nauk SSSR},
   volume={280},
   date={1985},
   pages={563--565},
   issn={0002-3264},
   review={\MR{775926}},
 }



\end{biblist}
\end{bibdiv}
\end{document}